\documentclass[a4paper,12pt]{amsart}

\usepackage{amsmath,amssymb,amsfonts,amsthm}
\usepackage{enumerate}
\usepackage{tikz}
\usepackage[colorlinks=true,pdftex]{hyperref}
\usetikzlibrary{calc,through,backgrounds}

\allowdisplaybreaks[1]%allows to break equations if really necessary
\newtheorem{theorem}{Theorem}
\newtheorem{prop}[theorem]{Proposition}
\newtheorem{lemma}[theorem]{Lemma}
\newtheorem{corollary}[theorem]{Corollary}
\newtheorem{defn}[theorem]{Definition}
\newtheorem{remark}[theorem]{Remark}
\newtheorem{example}[theorem]{Example}

\title[Kazhdan--Lusztig polynomials for $(B_N,A_{N-1})$]{Kazhdan--Lusztig 
polynomials for the Hermitian symmetric pair $(B_N,A_{N-1})$}
\author[K.~Shigechi]{Keiichi~Shigechi}
\email{k1.shigechi at gmail.com}
\date{\today}
\newcommand\linkpattern[5]{
\draw(0.8,0)--(#1+0.2,0);
\foreach \x/\y in {#2}
\draw[thick](\x,0)..controls (\x,0.4*\y-0.4*\x) 
		and (\y,0.4*\y-0.4*\x)..(\y,0);
\foreach \x/\y in {#3}
	\draw[thick](\x,0)--(\x,1)node[anchor=south]{\rm{\y}};
\foreach \x/\y in {#4}
	\draw[thick](\x,0)--(\x,1)
	node[circle,inner sep=1pt,draw,anchor=south]{\rm\y};
\foreach \x/\y in {#5}
	\draw[thick](\x,0)--(\x,1)(\y,0)--(\y,1);
\foreach \x/\y in {#5}
	\draw[thick,dotted](\x,1)--(\y,1);
}
\newcommand\tikzpic[1]{
\raisebox{-0.3\totalheight}{
\begin{tikzpicture}
#1
\end{tikzpicture}
}}

\begin{document}
\begin{abstract}
We provide combinatorial rules to compute Kazhdan--Lusztig polynomials 
for the Hermitian symmetric pair $(B_N,A_{N-1})$ when the Hecke algebra 
has unequal parameters.
They are obtained by filling regions delimited by paths with ballot 
strips.
We also extend the binary tree algorithm introduced by Lascoux and  
Sch\"utzenberger to our case.
\end{abstract}

\maketitle

\section{Introduction}
In \cite{KL79}, Kazhdan and Lusztig introduced 
{\it Kazhdan--Lusztig polynomials} $P_{x,y}$
indexed by two elements $x$ and $y$ of an arbitrary Coxeter group. 
These polynomials are the coefficients of the change of basis from the 
standard bases of the Hecke algebra to the Kazhdan--Lusztig bases. 
They play an important role in various research fields such as 
algebraic combinatorics~\cite{KL80-1}, topology of Schubert varieties~\cite{KL80-2} 
and representation theory of Verma modules~\cite{BB81,BK81}. 
In \cite{Deo87}, Deodhar introduced the concept of parabolic Kazhdan--Lusztig 
polynomials $P^{\pm}_{\alpha,\beta}$ for a Coxeter group. 
They are associated with the induced representation of the Hecke algebra by the 
one-dimensional representations of parabolic subgroups.
One of the important examples is the one with the Weyl group of type $A$ 
(the symmetric group $S_N$) and the maximal parabolic subgroup 
$S_K\times S_{N-K}$. 
This example has been studied as Kazhdan--Lusztig polynomials for 
Grasmannian permutations \cite{Bre02,LS81}. 
In \cite{SZJ10}, we provided a unified treatment of maximal parabolic 
Kazhdan-Lusztig polynomials in the language of paths.
In this paper, %%along the spirit of \cite{SZJ10}, 
we continue to investigate 
the Kazhdan-Lusztig polynomials for the Hermitian symmetric pair 
$(B_N,A_{N-1})$.

We have two types of parabolic Kazhdan--Lusztig polynomials 
$P_{\alpha,\beta}^\pm$ due to the choice of the projection map from 
$\mathbb{C}[S_N^C]$ to $\mathbb{C}[S_N^C/S_N]$ (see Section 2). 
In \cite{Boe}, Boe gave a combinatorial description of the Kazhdan-Lusztig
polynomials $P_{\alpha,\beta}^+$ for Hermitian symmetric pairs.
He generalized the binary tree algorithm introduced by Lascoux and 
Sch\"utzenberger keeping its flavour. 
On the other hand in \cite{Bre09}, the analysis for $P_{\alpha,\beta}^-$ was 
done by using the concept of shifted Dyck partitions. 
We can identify these cases with the maximally parabolically induced modules of 
the Hecke algebra with equal Hecke parameters.
One of the main results of this paper is to give a combinatorial description of 
the Kazhdan-Lusztig basis and polynomials $P_{\alpha,\beta}^\pm$ for the Hecke
algebra of type $B_N$ with unequal Hecke parameters for the Hermitian symmetric
pair $(B_N,A_{N-1})$.
We provide a unified treatment of $P_{\alpha,\beta}^\pm$ for the unequal parameter 
cases and give a generalization of Boe's algorithm to compute them. 

One way of an analysis of Kazhdan--Lusztig polynomials is to solve 
the recurrence relations of the $R$-polynomials, which have the same 
information as Kazhdan-Lusztig polynomials~\cite{Bre02,Bre09}.
In this paper, we study Kazhdan--Lusztig polynomials by using combinatorial
properties of the Hermitian symmetric pair $(B_N,A_{N-1})$, namely 
those of the coset space $S_N^C/S_N$ where $S_N$ and $S_N^C$ are the Weyl group
of type A and C. 
Our analysis has the flavour of the concept of tangles and link patterns 
used in statistical mechanics and that of Temperley--Lieb algebra 
\cite{BKW76,Mar90,Sal89,TL71}. 
We introduce the ballot strips (similar to shifted Dyck partitions 
in \cite{Bre09}) and graphical rules to compute two types of generating 
functions $Q^\pm_{\alpha,\beta}$ in a similar way as \cite{SZJ10}. 
These generating functions are shown to be equal to Kazhdan--Lusztig 
polynomials, that is, $Q^\pm_{\alpha,\beta}=P_{\alpha,\beta}^\pm$.

The plan of the paper is as follows. 
In Section \ref{sec-KL}, we introduce Kazhdan--Lusztig polynomials and their parabolic 
analogues. 
Then, we explain their inversion relations.
In Section \ref{sec-comb}, we introduce a concept of ballot strips and new diagrammatic 
rules $0, I$ and $II$ to stack these strips in a skew Ferrers diagram.
After defining generating functions $Q^\pm_{\alpha,\beta}$ for stacking
of strips, we provide the inversion relations for $Q^\pm_{\alpha,\beta}$.
Section \ref{sec-KL2} is devoted to the analysis of Kazhdan--Lusztig 
polynomials $P_{\alpha,\beta}^-$.  
The point is that we are able to compute $P^-_{\alpha,\beta}$ directly 
through link patterns.
Together with the inversion formula for $Q^\pm_{\alpha,\beta}$, 
we show $Q^\pm_{\alpha,\beta}=P^\pm_{\alpha,\beta}$.
The factorization property of the Kazhdan--Lusztig basis is presented.
In Section \ref{sec-binary}, we generalize the binary tree algorithm 
introduced in \cite{Boe,LS81}. 
This gives an alternative combinatorial algorithm for the 
computation of $P^+_{\alpha,\beta}$. 
When the two Hecke parameters are equal, {\it i.e.,} $t_N=t$ (see 
Section \ref{sec-KL} for notations), the algorithm coincides with Boe's. 
Further, the generating function $Q^+_{\alpha,\beta}$ introduced in Section \ref{sec-comb} 
is shown to be equal to the generating function of a generalized binary tree.

\paragraph{Notations} we denote by $\mathbb{N}_+$ the set of positive 
integers and $\mathbb{N}:=\mathbb{N}_+\cup\{0\}$. 
The $t$-deformed integers are $[m]:=(t^m-t^{-m})/(t-t^{-1})$ for $m\in\mathbb{Z}$.
Also $\langle0\rangle:=1$ and $\langle m\rangle:=t^m+t^{-m}$ for 
$m\in\mathbb{N}_+$.

\section{Kazhdan--Lusztig polynomials}
\label{sec-KL}
\subsection{Definitions}
Let $S_N^C$ be the finite Weyl group associated with the Dynkin diagram of type
$C$ and generated by $s_i, 1\le i\le N$ with defining relations $s_i^2=1$ for 
$1\le i\le N$, $(s_is_{i+1})^3=1$ for $1\le i\le N-2$ and $(s_{N-1}s_N)^4=1$. 
Let $w=s_{i_1}\ldots s_{i_r}\in S_N^C$ be a reduced word. 
The {\it length} functions $l,l',l_N:S_N^C\rightarrow\mathbb{N}$ is defined 
as $l'(w):=\mathrm{Card}\{i_j:1\le i_j\le N-1\}$, 
$l_N(w):=\mathrm{Card}\{i_j: i_j=N\}$ and $l(w):=l'(w)+l_N(w)=r$. 
The symmetric group $S_N$ of $N$ letters is a subgroup of $S_N^C$. 
The restriction of $l$ on $S_N$ is the standard length function of $S_N$. 

We use a natural partial order in $S_N^C$, known as the (strong) {\it Bruhat 
order}. 
For a given reduced word $w=s_{i_1}\ldots s_{i_r}$, a {\it subexpression} of 
$w$ is of the form $s_{j_1}\ldots s_{j_q}$ (or empty) with 
$1\le j_i<j_2<\ldots<i_q\le r$. 
Then, $w'\le w$ if and only if $w'$ can be obtained as a subexpression of a 
reduced expression of $w$. 

The Iwahori--Hecke algebra $\mathcal{H}:=\mathcal{H}(S_N^C)$ of type $B_N$ 
is an unital, associative algebra over $\mathbb{C}[t^{\pm1},t_N^{\pm1}]$ 
satisfying 
\begin{eqnarray*}
&&(T_i-t)(T_i+t^{-1})=0, \ \ 1\le i\le N-1, \\	
&&(T_N-t_N)(T_N+t_N^{-1})=0, \\
&&T_iT_{i+1}T_i=T_{i+1}T_{i}T_{i+1}, \ \ 1\le i\le N-2, \\
&&T_{N-1}T_{N}T_{N-1}T_{N}=T_{N}T_{N-1}T_{N}T_{N-1}.
\end{eqnarray*}
The set $\{T_w\}_{w\in S_N^C}$ is the standard monomial basis of $\mathcal{H}$.
Throughout this paper, we consider the two cases for the Hecke parameters 
$(t,t_N)$:
\begin{enumerate}[\bf C{a}se A:]
\item $t$ and $t_N$ are algebraically independent,
\item $t_N=t^m$ with some $m\in\mathbb{N}_+$.
\end{enumerate}
We define
\begin{eqnarray*}
\mathbf{t}^{l(w)}:=\left\{
\begin{array}{cc}
t^{l'(w)}t_N^{l_N(w)} & \mathrm{for\  Case A} \\
t^{l'(w)+ml_N(w)} & \mathrm{for\  Case B}
\end{array}\right..
\end{eqnarray*}
For $v,w\in S_{N}^C$, we denote $\mathbf{t}^{l(v)}/\mathbf{t}^{l(w)}$ 
by $\mathbf{t}^{l(v)-l(w)}$.

We define the bar involution of $\mathcal{H}$, $\mathcal{H}\ni a\mapsto\bar{a}$ 
by $T_i\mapsto T_i^{-1}$, $1\le i\le N$, together with $t^p\mapsto t^{-p}$ 
for $p\in\mathbb{N}_+$ (for Case A \& B) and $t_N\mapsto t_N^{-1}$ (for 
Case A only). 

We consider the abelian groups $\Gamma^A=\{t^it_N^j|i,j\in\mathbb{Z}\}$
and $\Gamma^B=\{t^i|i\in\mathbb{Z}\}$ for Case A and B respectively.
The lexicographic order of $\Gamma^X$ ($X=A,B$) is defined by
$\Gamma^X=\Gamma_+^X\cup\{1\}\cup\Gamma^X_-$ ($X=A,B$) where
\begin{eqnarray*}
\Gamma^A_+&:=&\{t^it_N^j|i>0,j\in\mathbb{Z}\}\cup\{t_N^i|i>0\}, \\
\Gamma^B_+&:=&\{t^i|i>0\}.
\end{eqnarray*}

Then we have 
\begin{theorem}[Lusztig~\cite{Lus03}]
There exists a unique basis $\{C^{X}_w:w\in S_N^C\}$ of $\mathcal{H}$ and 
a polynomial $P^{X}_{v,w}$ such that $\overline{C_w^X}=C^X_w$ and 
\begin{eqnarray}
C^{X}_w=\sum_{v\le w}\mathbf{t}^{l(v)-l(w)}P^{X}_{v,w}T_v
\end{eqnarray}
where $\mathbf{t}^{l(v)-l(w)}P^{X}_{v,w}\in\mathbb{Z}(\Gamma^X_-)$ ($X=A,B$). 
\end{theorem}

\subsection{The coset space}
Let $W^N$ be the left coset space $S_N^C/S_N$. 
The following objects are bijective to each other:
\begin{enumerate}[(i)]
\item A binary string $\alpha\in\{1,2\}^N$. Let $\mathcal{P}_N$ be 
the set of binary strings in $\{1,2\}^N$.
\item A path from $(0,0)$ to $(N,n)$ with $|n|\le N$ and 
$N-n\in2\mathbb{Z}$, where each step is in the direction $(1,\pm1)$. 
%%We also denote by $\mathcal{P}_N$ the set of paths.
\item A minimal (and maximal) representative of the coset $W^N$.
\item A {\it shifted Ferrers diagram} specified by a path $\alpha$.
\end{enumerate}
Before proceeding to show explicit bijections, we introduce some terminologies.
For later convenience, we introduce a sign $\epsilon\in\{+,-\}$. 
When $\epsilon=+$ (resp. $\epsilon=-$), we consider the maximal (resp. minimal)
representatives of the coset $W^N$.

\begin{defn}
For each binary string $\alpha\in\mathcal{P}_N$, we denote by 
$w^\epsilon(\alpha)$ and $\lambda^\epsilon(\alpha)$ the representative in 
$W^N$ and a (rotated) shifted Ferrers diagram corresponding to $\alpha$.
\end{defn}

A bijection between (i) and (ii). 
Let $\alpha\in\mathcal{P}_N$ be a binary string. 
A path starts from $(0,0)$ and each move is in the direction $(1,\epsilon)$
if $\alpha_i=1$ or $(1,-\epsilon)$ if $\alpha_i=2$.
Reversely, for a given path, we can read $\alpha_i$ according to the tangent of 
each step.
Hereafter, we identify a path with a binary string.

\begin{remark}
Note that the binary string corresponding to a given path $\alpha$ depends on 
the choice of the sign $\epsilon$. 
Suppose that a string $\alpha_+$ for $\epsilon=+$ is associated with the path 
$\alpha$. 
Then, the string $\alpha_-$ for $\epsilon=-$ is obtained by exchanging $1$ 
and $2$ in $\alpha_+$. 
\end{remark}

A bijection between (ii) and (iv).
Let $\alpha\in\mathcal{P}_N$ be a path. 
For $\epsilon=+$, consider the set of integral points
\begin{eqnarray}
S^+(\alpha):=\{(i,j): 
(i,j) \mathrm{\ is\ above\ the\  path\ } \alpha, 
0<i\le N, |j|\le i,  i+j-1\in2\mathbb{Z} \}.
\label{defS}
\end{eqnarray}
We put (45 degree rotated) squares of length $\sqrt{2}$ whose center are 
all points in $S^+(\alpha)$.
The set of squares can be regarded as (45 degree rotated) shifted Ferrers 
diagram $\lambda^+(\alpha)$.
For $\epsilon=-$, we define $S^-(\alpha)$ by replacing ``above" by ``below" 
in the set $S^+(\alpha)$ and define $\lambda^-(\alpha)$ similarly.

We call a box $(N,j)\in\lambda/\mu$ for some $j$ an {\it anchor} box.

Let $\alpha\in\mathcal{P}_N$ be a path and $\lambda$ be the associated 
Ferrers diagram $\lambda^\epsilon(\alpha)$.
We denote by $|\lambda|$ the number of boxes in the skew Ferrers diagram $\lambda$.
By abuse of notation, we also denote $|\alpha|:=|\lambda(\alpha)|$. 
Notice that $|\alpha|$ depends on the sign $\epsilon$ and we omit $\epsilon$ 
when it is obvious.

A bijection between (i) and (iii). 
A bijection directly follows from \cite{Proc}.
We fix the convention by assigning the binary string $1^N$ to the identity 
in $W^N$ for $\epsilon=\pm$. 
A reduced word $w^+(\alpha)$ (resp. $w^-(\alpha)$) is obtained from 
$\lambda^+(\alpha)$ (resp. $\lambda^-(\alpha)$) as follows.
Starting from the top (resp. bottom) box, we read the boxes left downward 
(resp. upward) according to the column of $\lambda^+(\alpha)$ 
(resp. $\lambda^-(\alpha)$).
If the number of boxes in the column is $k_1$, we assign an ordered product 
$s_{N-k_1+1}\ldots s_N$ to this column.
Then, move to the next column with $k_2$ boxes.
Continue until all columns are visited.
Therefore, $w^+(\alpha)$ (resp. $w^-(\alpha)$) is of the form 
$(s_{N-k_r+1}\ldots s_N)\ldots (s_{N-k_2+1}\ldots s_N)(s_{N-k_1+1}\ldots s_N)$
with $1\le r\le N$ and $k_1>k_2>\ldots k_r\ge1$.
We denote this ordered product by 
\begin{eqnarray*}
w^\pm(\alpha)=
\prod_{(i,j)\in\lambda^\pm(\alpha)}^{\leftarrow}s_i.
\end{eqnarray*}

Let us take two paths $\alpha,\beta\in\mathcal{P}_N$ and fix 
the sign $\epsilon$.
We denote by $\alpha\le\beta$ when corresponding representatives
satisfy $w^\epsilon(\alpha)\le w^\epsilon(\beta)$.
Note that when $\alpha$ is above $\beta$, $\alpha<\beta$ (resp. 
$\alpha>\beta$) for $\epsilon=+$ (resp. $\epsilon=-$). 

\begin{example}
Let $\alpha=221121$ and $\epsilon=+$. 
The shifted Ferrers diagram $\lambda^+(\alpha)$ is shown below.
The path $\alpha$ is the lowest path from $O$ to $B$ and the path $111111$
is the up-right one from $O$ to $A$.
When $\epsilon=-$, the binary string for the path from $O$ to $B$ is $112212$.
As a maximal representation in $W^N$, 
$w^+(\alpha)=s_5s_6 s_2s_3s_4s_5s_6 s_1s_2s_3s_4s_5s_6$. 
The boxes with $*$ are anchor boxes.
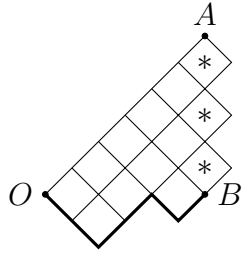
\begin{center}
\begin{tikzpicture}[scale=0.35]
\draw[very thick] (0,0)node[anchor=east]{$O$}--(2,-2)--(4,0)
--(5,-1)--(6,0)node[anchor=west]{$B$};
\draw (0,0)--(6,6)node[anchor=south]{$A$};
\draw (1,1)--(3,-1) (2,2)--(4,0) (3,3)--(6,0) 
(4,4)--(7,1)--(6,0) (5,5)--(7,3)--(4,0) 
(6,6)--(7,5)--(1,-1);
\filldraw (0,0)circle(3pt)(6,0) circle (3pt)(6,6) circle (3pt);
\foreach \x in {1,3,5} \draw (6,\x)node{$*$}; 
\end{tikzpicture}
\end{center}
\end{example}

\subsection{Parabolic Kazhdan--Lusztig polynomials}
An element $w\in S_N^C$ is uniquely written as $w=xw'$ such that $x\in W^N$
and $w'\in S_N$.
The projection $\varphi:S_N^C\rightarrow W^N$ induces two natural projection
maps $\varphi^\pm:\mathcal{H}\cong\mathbb{C}[S_N^C]\rightarrow\mathbb{C}[W^N]$, 
$T_w\mapsto(\pm t^{\pm1})^{l(w')}m_{\varphi(w)}$, where $\{m_w\}_{w\in W^N}$
is the standard basis of $\mathbb{C}[W^N]$. 
We require that $\varphi^\pm$ commute with the action of $\mathcal{H}$. 

Let $\alpha:=\alpha_1\alpha_2\ldots\alpha_N\in\mathcal{P}^N$ be a binary 
string and $\mathcal{M}^\pm:=\mathbb{C}[W^N]$ be a vector space spanned 
by $\langle m_{\alpha}:\alpha\in\mathcal{P}^N\rangle$. 
A simple transposition $s_i\in S_N^C$ naturally acts on the binary string
$\alpha$, {\it i.e.}, $s_i\cdot\alpha=\ldots\alpha_{i+1}\alpha_i\ldots$, 
$1\le i\le N-1$, and $s_N\cdot\alpha=\ldots(3-\alpha_N)$.  
The action of $\mathcal{H}$ on the module $\mathcal{M}^\epsilon$ with 
$\epsilon\in\{+,-\}$ is given by
\begin{eqnarray}
T_im_\alpha&=&\left\{ 
\begin{array}{cc}
\epsilon t^\epsilon m_\alpha & \alpha_i=\alpha_{i+1}, \\
m_{s_i\cdot\alpha}  & \alpha_i<\alpha_{i+1}, \\
m_{s_i\cdot\alpha}+(t-t^{-1})m_{\alpha}  & \alpha_{i+1}<\alpha_i,  
\end{array}\right.
\ \  \mathrm{for\ } 1\le i\le N-1,  \\
T_Nm_\alpha&=&\left\{ 
\begin{array}{cc}
m_{s_N\cdot\alpha}  & \alpha_N=1, \\
m_{s_N\cdot\alpha}+(t_N-t_N^{-1})m_{\alpha}  & \alpha_N=2,  
\end{array}\right.
\end{eqnarray} 
for both Case A and B.

We introduce parabolic analogue of the Kazhdan--Lusztig basis:
\begin{theorem}[Deodhar~\cite{Deo87}]
There exists a unique basis $\{C_x^\pm\}_{x\in W^N}$ of $\mathcal{M}^\pm$ 
such that $C_{x}^\pm=\overline{C_{x}^\pm}$ and 
\begin{eqnarray*}
C_y^\pm=\sum_{x\le y}\mathbf{t}^{l(x)-l(y)}P_{x,y}^{\pm}m_x
\end{eqnarray*}
where $P_{y,y}^{\pm}=1$ and 
$\mathbf{t}^{l(x)-l(y)}P_{x,y}^{\pm}\in\mathbb{Z}(\Gamma^X_-)$ 
for Case X (X=A,B).
\end{theorem}

Hereafter, we denote by $P_{x,y}^{A,\pm}$ (resp. $P_{x,y}^{B,\pm}$) the parabolic 
Kazhdan--Lusztig polynomials for Case A (resp. Case B).

\begin{theorem}
Let $X\in\{A,B\}$. We have the inversion formula for $P^{X,\pm}_{\alpha,\beta}$:
\begin{eqnarray}
\sum_{\alpha\in\mathcal{P}_N}(-1)^{|\alpha|+|\beta|}
P^{X,-}_{\alpha,\beta}P^{X,+}_{\alpha,\gamma}=\delta_{\beta,\gamma}.
\label{invKL}
\end{eqnarray}
\end{theorem}
%%%
\begin{proof}
The relations among $P^\pm_{\alpha,\beta}$ and the original Kazhdan--Lusztig
polynomials $P_{x,y}$ are given by Proposition 3.4 and Remark 3.8 in \cite{Deo87}.
Together with the inversion formula given in Theorem 3.1 in \cite{KL79}, we have 
Eqn.(\ref{invKL}). For details, see also Theorem~4 in \cite{SZJ10}.
\end{proof}

\section{Combinatorics}
\label{sec-comb}
\subsection{Ballot strips}
A {\it ballot path} of length $(l,l')\in\mathbb{N}^2$ is a path from 
$(x,y)\in\mathbb{Z}^2$ to $(x+2l+l',y+l')$ and over the horizontal line $y$.
 
A {\it ballot strip} of length $(l,l')\in\mathbb{N}^2$ is obtained by putting
unit boxes (45 degree rotated) whose center are at the vertices of a ballot 
path of length $(l,l')$ (see some examples on Fig.\ref{Ballotstrip}).
Note that a single box (corresponding to the length $(l,l')=(0,0)$) is also 
included as a ballot strip.
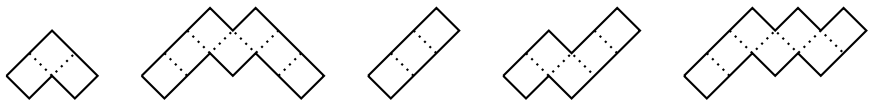
\begin{figure}[ht]
\begin{tikzpicture}[scale=0.3]
\draw[thick](0,0)--(1,-1)--(2,0)--(3,-1)--(4,0)--(2,2)--(0,0);
\draw[thick,dotted](1,1)--(2,0)--(3,1);
\end{tikzpicture}
\quad
\begin{tikzpicture}[scale=0.3]
\draw[thick](0,0)--(1,-1)--(3,1)--(4,0)--(5,1)--(7,-1)--
(8,0)--(5,3)--(4,2)--(3,3)--(0,0);
\draw[thick,dotted](1,1)--(2,0)
(2,2)--(3,1)--(4,2)--(5,1)--(6,2)(7,1)--(6,0);
\end{tikzpicture}
\quad
\begin{tikzpicture}[scale=0.3]
\draw[thick](0,0)--(1,-1)--(4,2)--(3,3)--(0,0);
\draw[thick,dotted](1,1)--(2,0)(2,2)--(3,1);
\end{tikzpicture}
\quad
\begin{tikzpicture}[scale=0.3]
\draw[thick](0,0)--(1,-1)--(2,0)--(3,-1)--(6,2)--(5,3)--(3,1)--(2,2)--(0,0);
\draw[thick,dotted](1,1)--(2,0)--(3,1)--(4,0)(4,2)--(5,1);
\end{tikzpicture}
\quad
\begin{tikzpicture}[scale=0.3]
\draw[thick](0,0)--(1,-1)--(3,1)--(4,0)--(5,1)--
(6,0)--(8,2)--(7,3)--(6,2)--(5,3)--(4,2)--(3,3)--(0,0);
\draw[thick,dotted](1,1)--(2,0)(2,2)--(3,1)--(4,2)--(5,1)--(6,2)--(7,1);
\end{tikzpicture}
\caption{Examples of ballot strips: The length is 
$(1,0), (3,0), (0,2), (1,2)$ and $(2,2)$ from left.}
\label{Ballotstrip}
\end{figure}

\begin{remark}
A ballot path of length $(l,0)$ is nothing but a Dyck strip in \cite{SZJ10}.
\end{remark}

Hereafter, a box $(x,y)$ means a unit box whose center is $(x,y)$.
Let $b$ a box $(x,y)$.
Four boxes $(x\pm1,y\pm1)$ are neighbours of $b$.
The box $(x+1,y+1)$ is said to be NE (north-east) of $b$ and similarly
the other three boxes are NW, SW and SE of $b$.
The two boxes $(x,y\pm2)$ are said to be {\it just above} or 
{\it just below} $b$.

Recall the definition of an anchor box in a skew Ferrers diagram.
We put a constraint for a ballot strip as follows.

\begin{enumerate}
\item[\bf Rule 0] (Case A and B): The rightmost box of a ballot strip 
of length $(l,l')$ with $l'\ge1$ is on an anchor box.
\end{enumerate}
Let $\mathcal{D,D'}$ be ballot strips.
We define two rules to pile $\mathcal{D'}$ on top of $\mathcal{D}$ in addition
to Rule 0.

\begin{enumerate}[\bf Rule I:]
\item 
\begin{enumerate}[(a)]
\item Case A \& B: If there exists a box of $\mathcal{D}$ just below a box 
of $\mathcal{D'}$, then all boxes just below a box of $\mathcal{D'}$ belong 
to $\mathcal{D}$.
\item Case B: Suppose $l'\ge m$. The number of ballot strips of length 
$(l,l')$ is even for $l'-m\in2\mathbb{Z}$, and zero for otherwise.
\end{enumerate}
\item 
\begin{enumerate}[(a)]
\item Case A\& B: If there exists a box of $\mathcal{D'}$ just above, NW,
or NE of a box of $\mathcal{D}$, then all boxes just above, NW, 
and NE of a box of $\mathcal{D}$ belong to $\mathcal{D}$ or $\mathcal{D'}$.
\item Case B: Suppose $l'\ge m$. If there exists a ballot strip $\mathcal{D}$
of length $(l,l')$ with $l'-m\in2\mathbb{Z}$ (resp. $l'-m-1\in2\mathbb{Z}$), 
then there is a strip of length $(l'',l'+1), l''\ge l$ (resp. $(l'',l'-1)$, $l''\le l$) 
just above (resp. just below) $\mathcal{D}$.
\end{enumerate}
\end{enumerate}

Roughly, Rule I (resp. Rule II) means that we are allowed to pile 
ballot strips of smaller or equal (resp. longer) length on top of 
a ballot strip (see Figure \ref{Ballotpile}).

\begin{figure}[ht]
\begin{tikzpicture}[scale=0.3]
\draw(0,0)--(3,-3)--(7,1)--(8,0)--(9,1)--(10,0)--(11,1)--(14,-2)--(17,1)--(16,2)
(2,-2)--(7,3)--(8,2)--(9,3)--(10,2)--(11,3)--(14,0)--(17,3)--(16,4)
(1,-1)--(7,5)--(8,4)--(9,5)--(10,4)--(11,5)--(14,2)
(13,1)--(17,5)--(15,7)
(0,0)--(7,7)--(9,5)--(11,7)--(15,3)
(1,1)--(3,-1)(2,2)--(4,0)(3,3)--(5,1)(4,4)--(5,3)(5,5)--(6,4)
(8,6)--(9,7)--(10,6)
(13,3)--(17,7)--(16,8)(12,4)--(16,8)(14,6)--(15,5); 
\end{tikzpicture}
\qquad
\begin{tikzpicture}[scale=0.3]
\draw(0,0)--(4,-4)--(5,-3)--(6,-4)--(8,-2)--(10,-4)--(11,-3)--(12,-4)
--(17,1)--(16,2)
(3,-3)--(4,-2)--(5,-3)--(6,-2)--(7,-3)
(9,-3)--(11,-1)--(13,-3)
(2,-2)--(4,0)--(5,-1)--(6,0)--(8,-2)--(11,1)--(14,-2)
(1,-1)--(4,2)--(5,1)--(6,2)--(8,0)--(11,3)--(14,0)--(17,3)--(16,4)
(0,0)--(4,4)--(5,3)--(6,4)--(8,2)--(11,5)--(14,2)--(16,4);
\end{tikzpicture}

\caption{Examples of stacks of ballot strips satisfying Rule I (left)
and Rule II (right).}
\label{Ballotpile}
\end{figure}
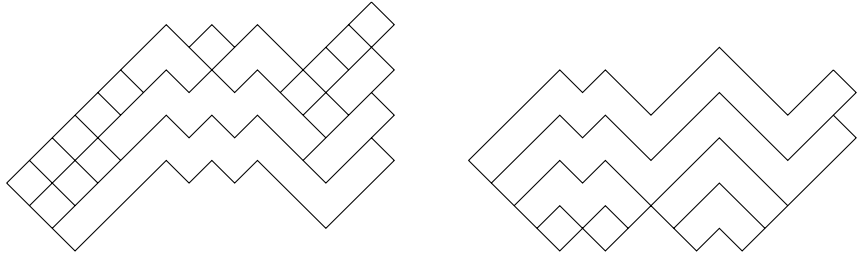

\subsection{Generating function}
Let $\alpha,\beta$ be paths in $\mathcal{P}_N$ such that $\alpha<\beta$.
These two paths characterize the domains, namely the skew Ferrers diagram
$\lambda^\epsilon(\beta)/\lambda^\epsilon(\alpha)$. 
We fill these domains with ballot strips. 
We denote by $\mathrm{Conf}(\alpha,\beta)$ the set of all such possible 
configurations of ballot strips, and by $\mathrm{Conf}^{I/II}(\alpha,\beta)$
the subset of configurations satisfying Rule I/II.

Let $\mathcal{D}$ be a ballot strip of length $(l,l')\in\mathbb{N}^2$.
We denote by $\mathrm{wt}^X(\mathcal{D})$, $X\in\{A,B\}$, the weight 
for a ballot strip $\mathcal{D}$, which is given by
\begin{enumerate}[C{a}se A:]
\item 
\begin{eqnarray}
\mathrm{wt}^A(\mathcal{D}):=\left\{
\begin{array}{cc}
t^{2l+l'}, & l' \mathrm{\ is\ even}, \\
-\sigma t^{2l+l'-1}t_N^2, &  l' \mathrm{\ is\ odd}.	 
\end{array}\right.
\label{wtA}
\end{eqnarray}
\item
\begin{eqnarray}
\mathrm{wt}^B(\mathcal{D}):=\left\{
\begin{array}{cc}
\sigma^{l'}t^{2l+2l'}, & 0\le l'\le m-1 \\
t^{m+l'+2l}, &  l'\ge m, l'-m\in2\mathbb{Z} ,  \\
t^{m+l'+2l-1}, &  l'\ge m, l'-m-1\in2\mathbb{Z},  \\
\end{array}\right.
\label{wtB}
\end{eqnarray}
where the sign $\sigma=+$ (resp. $-$) in the case of Rule I (resp. Rule II).
\end{enumerate}

\begin{defn}
The generating function of ballot strips for the paths $\alpha<\beta$ is 
defined by 
\begin{eqnarray*}
Q_{\alpha,\beta}^{X,Y,\epsilon}
=\sum_{\mathcal{C}\in\mathrm{Conf}^X(\alpha,\beta)}
\prod_{\mathcal{D}\in\mathcal{C}}\mathrm{wt}^X(\mathcal{D}),
\end{eqnarray*}
where $X\in\{A,B\}, Y\in\{I,II\}$ and $\epsilon\in\{+,-\}$.
\end{defn}

Note that $\mathrm{Conf}^{II}(\alpha,\beta)$ has at most one configuration.
Recall that when two paths satisfy $\alpha<\beta$ with the sign $\epsilon$, 
the change of the sign $\epsilon\mapsto-\epsilon$ yields $\alpha>\beta$.
Therefore, we have 
\begin{eqnarray*}
Q^{X,Y,\epsilon}_{\alpha,\beta}
=Q^{X,Y,-\epsilon}_{\beta,\alpha}.
\end{eqnarray*}

\begin{example}
Let $(\alpha,\beta)=(111111,211212)$ and $\epsilon=+$.
The possible configurations of ballot strips for Case A and Case B 
($m\ge2$) are
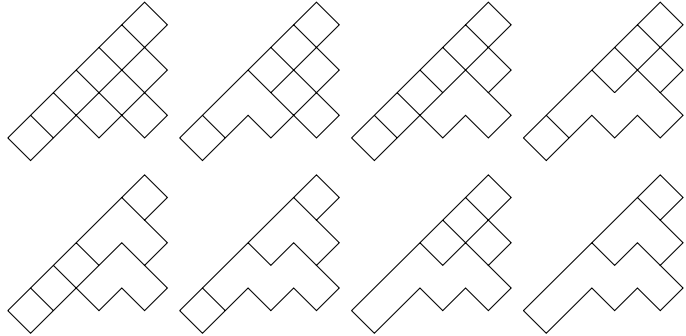
\begin{center}
\begin{tikzpicture}[scale=0.3]
\draw(0,0)--(1,-1)--(3,1)--(4,0)--(5,1)--(6,0)--(7,1)--(6,2)--(7,3)--(6,4)
--(7,5)--(6,6)--(0,0)
(1,1)--(2,0)(2,2)--(3,1)--(6,4)(3,3)--(5,1)--(6,2)(4,4)--(6,2)(5,5)--(6,4);
\end{tikzpicture}
\begin{tikzpicture}[scale=0.3]
\draw(0,0)--(1,-1)--(3,1)--(4,0)--(5,1)--(6,0)--(7,1)--(6,2)--(7,3)--(6,4)
--(7,5)--(6,6)--(0,0)
(1,1)--(2,0)(4,2)--(6,4)(3,3)--(5,1)--(6,2)(4,4)--(6,2)(5,5)--(6,4);
\end{tikzpicture}
\begin{tikzpicture}[scale=0.3]
\draw(0,0)--(1,-1)--(3,1)--(4,0)--(5,1)--(6,0)--(7,1)--(6,2)--(7,3)--(6,4)
--(7,5)--(6,6)--(0,0)
(1,1)--(2,0)(2,2)--(3,1)--(6,4)(3,3)--(4,2)(4,4)--(6,2)(5,5)--(6,4);
\end{tikzpicture}
\begin{tikzpicture}[scale=0.3]
\draw(0,0)--(1,-1)--(3,1)--(4,0)--(5,1)--(6,0)--(7,1)--(6,2)--(7,3)--(6,4)
--(7,5)--(6,6)--(0,0)
(1,1)--(2,0)(4,2)--(6,4)(3,3)--(4,2)(4,4)--(6,2)(5,5)--(6,4);
\end{tikzpicture}

\begin{tikzpicture}[scale=0.3]
\draw(0,0)--(1,-1)--(3,1)--(4,0)--(5,1)--(6,0)--(7,1)--(6,2)--(7,3)--(6,4)
--(7,5)--(6,6)--(0,0)
(1,1)--(2,0)(2,2)--(3,1)--(5,3)--(6,2)(3,3)--(4,2)(5,5)--(6,4);
\end{tikzpicture}
\begin{tikzpicture}[scale=0.3]
\draw(0,0)--(1,-1)--(3,1)--(4,0)--(5,1)--(6,0)--(7,1)--(6,2)--(7,3)--(6,4)
--(7,5)--(6,6)--(0,0)
(1,1)--(2,0)(3,3)--(4,2)--(5,3)--(6,2)(5,5)--(6,4);
\end{tikzpicture}
\begin{tikzpicture}[scale=0.3]
\draw(0,0)--(1,-1)--(3,1)--(4,0)--(5,1)--(6,0)--(7,1)--(6,2)--(7,3)--(6,4)
--(7,5)--(6,6)--(0,0)
(3,3)--(4,2)--(6,4)(4,4)--(6,2)(5,5)--(6,4);
\end{tikzpicture}
\begin{tikzpicture}[scale=0.3]
\draw(0,0)--(1,-1)--(3,1)--(4,0)--(5,1)--(6,0)--(7,1)--(6,2)--(7,3)--(6,4)
--(7,5)--(6,6)--(0,0)
(3,3)--(4,2)--(5,3)--(6,2)(5,5)--(6,4);
\end{tikzpicture}
\end{center}

The generating functions are 
\begin{eqnarray*}
Q_{\alpha,\beta}^{A,I,+}
&=&1+2t^2+2t^4+t^6-t_N^2t^4-t_N^2t^6, \\
Q_{\alpha,\beta}^{B,I,+}
&=&(1+t^2)^2(1+t^4), \quad m\ge2.
\end{eqnarray*}

For Case B with $m=1$, the number of possible configurations is six (all ballot
strips are of length $(l,0)$). 
The generating function is 
\begin{eqnarray*}
Q_{\alpha,\beta}^{B,I,+}=1+2t^2+2t^4+t^6.
\end{eqnarray*}
\end{example}

\begin{theorem}[Inversion Formula]
Let $X\in\{A,B\}$. 
The generating functions $Q_{\alpha,\beta}^{X,Y,\epsilon}$ satisfy
\begin{eqnarray}
\sum_{\beta\in\mathcal{P}_N}
Q_{\alpha,\beta}^{X,I,-}Q_{\beta,\gamma}^{X,II,-}
(-1)^{|\beta|+|\gamma|}=\delta_{\alpha,\gamma}.
\end{eqnarray}
\label{invBallot}
\end{theorem}
\begin{proof}
We refer to the proof of Theorem 5 in \cite{SZJ10} since we can apply the 
similar arguments to our case. 
Below, we give the outline of the proof and the difference from \cite{SZJ10}.

When $\alpha=\gamma$, the argument holds true. 
Let us fix two paths $\alpha,\gamma$ ($\alpha<\gamma$ in $\epsilon=-$) and 
a configuration of ballot strips 
$\mathcal{C}\in\mathrm{Conf}(\alpha,\gamma)$ such that there exists a path 
$\beta$ dividing $\mathcal{C}$ into two configurations $\mathcal{C}_I(\beta)$
and $\mathcal{C}_{II}(\beta)$ where 
$\mathcal{C}_I(\beta)\in\mathrm{Conf}^I(\alpha,\beta)$ and 
$\mathcal{C}_{II}(\beta)\in\mathrm{Conf}^{II}(\beta,\gamma)$.
Notice that $\beta$ depends on the configuration $\mathcal{C}$ and there may 
be several possible $\beta$'s.
We denote by $P(\mathcal{C})$ the set of such paths $\beta$'s for the 
configuration $\mathcal{C}$. 
The weight of the configuration $\mathcal{C}$ is given by 
$\mathrm{wt}(\mathcal{C})=\mathrm{wt}(\mathcal{C}_I)\mathrm{wt}(\mathcal{C}_{II})$
(recall the definitions of weights (\ref{wtA}) and (\ref{wtB})). 
Since we fix the configuration, the absolute value of the weight 
$|\mathrm{wt}(\mathcal{C})|$ is independent of the path $\beta$.
We denote by 
$\mathrm{wt}(\mathcal{C})=|\mathrm{wt}(\mathcal{C})|\mathrm{sign}(\mathcal{C})$, 
where $\mathrm{sign}(\mathcal{C})$ is a sign associated with the configuration 
$\mathcal{C}$. 
Therefore, we have 
\begin{eqnarray}
\sum_{\beta}Q_{\alpha,\beta}^{X,I,-}Q_{\beta,\gamma}^{X,II,-}
(-1)^{|\beta|+|\gamma|}
&=&\sum_{\mathcal{C}}\sum_{\beta\in P(\mathcal{C})}
\mathrm{wt}(\mathcal{C})(-1)^{|\beta|+|\gamma|}  \nonumber \\
&=&\sum_{\mathcal{C}}|\mathrm{wt}(\mathcal{C})|
\sum_{\beta\in P(\mathcal{C})}\mathrm{sign}(\mathcal{C})(-1)^{|\beta|+|\gamma|}.
\end{eqnarray} 
Below, we will show that 
$\sum_{\beta\in P(\mathcal{C})}\mathrm{sign}(\mathcal{C})
(-1)^{|\beta|+|\gamma|}=0$. 

Define the intersection
\begin{eqnarray*}
\mathcal{I}(\mathcal{C})
:=\left(
\bigcup_{\beta\in P(\mathcal{C})}\mathcal{C}_I(\beta)
\right)
\cap
\left(
\bigcup_{\beta\in P(\mathcal{C})}\mathcal{C}_{II}(\beta)
\right).
\end{eqnarray*}
The choice of $\beta$ determines that an element of 
$\mathcal{I}(\mathcal{C})$ belongs to $\mathcal{C}_I$ or 
$\mathcal{C}_{II}$.
Here an element $\mathcal{B}\in\mathcal{I}(\mathcal{C})$ 
means a ballot strip or ballot strips which are on top of each other
and glued together.
 
Let $\beta_1, \beta'_1$ be paths such that an element 
$\mathcal{B}\in\mathcal{I}(\mathcal{C})$ belongs to 
$\mathcal{C}_I(\beta_1)$ and $\mathcal{C}_{II}(\beta'_1)$, which
means that $\beta_1$ is above $\beta'_1$.
We show 
$\sum_{\beta=\beta_1,\beta'_1}\mathrm{sign}(\mathcal{B})
(-1)^{|\beta_1|-|\beta'_1|}=0$.
There are three cases for a possible element in $\mathcal{I}(\mathcal{C})$.

\paragraph{\bf Case 1}(Case A \& B)
Let $\mathcal{B}$ be a ballot strip of length $(l,0)$.
We need odd number of boxes to form the strip.
Therefore, the difference $|\beta_1|-|\beta'_1|$ is odd and the sign 
$\mathrm{sign}(\mathcal{B})=1$ for both $\epsilon=\pm$.
This implies the contributions of $\beta_1$ and $\beta'_1$ cancel each 
other.

\paragraph{\bf Case 2}(Case A only)
Let $\mathcal{B}$ be a ballot strip of length $(l,l')$ with $l'\ge1$.
When $l'$ is even, then the number of boxes to form the strip 
$\mathcal{B}$ is odd.
As in Case 1, the contributions of $\beta_1$ and $\beta'_1$ cancel.
When $l'$ is odd, the number of boxes to form $\mathcal{B}$ is even.
Therefore, $|\beta_1|-|\beta'_1|$ is even.
However, we have 
\begin{eqnarray*}
\mathrm{sign}(\mathcal{B})=\left\{
\begin{array}{cc}
+, & \mathcal{B}\in\mathcal{C}_{II}(\beta'_1), \\
-, & \mathcal{B}\in\mathcal{C}_{I}(\beta_1),
\end{array}
\right.
\end{eqnarray*}
which implies the contributions from $\beta_1$ and $\beta'_1$ cancel.

\paragraph{\bf Case 3}(Case B only)
Let $\mathcal{B}$ be a ballot strip of length $(l,l')$ with $l'\ge1$.
When $1\le l'\le m-1$, the contributions cancel as in Case 2.
Below, we assume that $l'\ge m$.
Suppose that $\mathcal{B}\in\mathcal{C}_I(\beta_1)$.
The length of $\mathcal{B}$ satisfies $l'-m\in2\mathbb{Z}$ due 
to Rule Ib. 
Further, there is another $\mathcal{B}'\in\mathcal{C}_I(\beta_1)$ 
such that the same length as $\mathcal{B}$, and two strips 
$\mathcal{B, B'}$ are on top of each other.
Without loss of generality, we assume that $\mathcal{B}'$ is just 
above $\mathcal{B}$.
We want to change the path $\beta_1$ to $\beta'_1$ such that 
$\mathcal{B}$ belongs to $\mathcal{C}_{II}(\beta'_1)$.
Since $\mathcal{B}\in\mathcal{C}_{II}(\beta'_1)$, $l'-m\in2\mathbb{Z}$ 
and Rule IIb, we need a strip $\mathcal{B}''$ of length $(l'',l'+1)$, 
$l''\ge l$, just above $\mathcal{B}$.
We will show that $\mathcal{B}''$ can be obtained from $\mathcal{B}'$ 
by gluing two ballot strips.
Let $b'$ be the leftmost box of $\mathcal{B}$ and $b$ be the northwest
box of $b'$. 
Since the strip $\mathcal{B'}$ is just above $\mathcal{B}$ and of the 
same length, the ballot strip $\mathcal{D}$ which contains $b$ 
satisfies $\mathcal{D}\in\mathcal{C}_{I}(\beta_1)$. 
%and $b\in\mathcal{C}_{II}(\beta'_1)$ simultaneously. 
Therefore, we are able to glue the strip $\mathcal{B}'$ and the strip 
$\mathcal{D}$ to form $\mathcal{B}''$.
We regard the region obtained by gluing $\mathcal{B}$ and $\mathcal{B}''$
(equivalently $\mathcal{B, B'}$ and $\mathcal{D}$) as an element of 
$\mathcal{I}(\mathcal{C})$. 
See Figure \ref{glue} for an example of this operation.
%%%
\begin{figure}[ht]
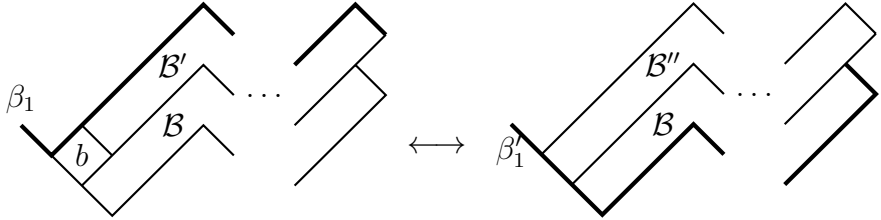

\begin{eqnarray*}
\tikzpic{[scale=0.4]
\draw[ultra thick](0,0)node[anchor=south]{$\beta_1$}--(1,-1)
--(6,4)--(7,3);
\draw[thick](1,-1)--(3,-3)--(6,0)--(7,-1)(2,0)--(3,-1)
(2,-2)--(6,2)--(7,1);
\draw(2,-1)node{$b$}(5,2)node{$\mathcal{B}'$}(5,0)node{$\mathcal{B}$};
\draw(8,1)node{$\ldots$};
\draw[ultra thick](9,2)--(11,4)--(12,3);
\draw[thick](9,0)--(12,3)(9,-2)--(12,1)--(11,2);
}
\longleftrightarrow
\tikzpic{[scale=0.4]
\draw[thick](0,0)node[anchor=north]{$\beta'_1$}--(1,-1)
--(6,4)--(7,3)(2,-2)--(6,2)--(7,1);
\draw[ultra thick](0,0)--(3,-3)--(6,0)--(7,-1)
;
\draw(5,2)node{$\mathcal{B}''$}(5,0)node{$\mathcal{B}$};
\draw(8,1)node{$\ldots$};
\draw[thick](9,2)--(11,4)--(12,3)(9,0)--(12,3);
\draw[ultra thick](9,-2)--(12,1)--(11,2);
}
\end{eqnarray*}
\caption{The region below (resp. above) the path $\beta_1$ (resp. $\beta'_1$) 
satisfies Rule I (resp. Rule II) in the left (resp. right) picture. 
The strip $\mathcal{B'}$ and the box $b$ (a strip of length $(0,0)$) form 
the strip $\mathcal{B}''$}
\label{glue}
\end{figure}
%%%
As a final step, we have to compare the weights of these strips.
It is clear from (\ref{wtB}) that the weight of $\mathcal{B}''$ is equal 
to the product of the weights $\mathcal{B}'$ and $\mathcal{D}$.
Since the total number of boxes in the strip $\mathcal{D}$  is odd, 
$|\beta_1|-|\beta'_1|$ is odd.
Further, the sign of this region is plus. 
These imply that the contributions of these two paths cancel.

In the above three cases, we change a path $\beta_1$ to $\beta'_1$
locally by involving one element in $\mathcal{I}(\mathcal{C})$.
We are also able to show the following two facts (see Proposition 4 
and Lemma 5 in \cite{SZJ10}): 1) When there exists a path 
$\beta\in P(\mathcal{C})$, $\mathcal{I}(C)$ is not empty. 
2) The distance of two elements in $\mathcal{I}(C)$ is at least two. 
From these two facts, there are $2^r$ possible paths in $P(\mathcal{C})$
when the cardinality of $\mathcal{I}(\mathcal{C})$ is $r$.
Therefore, the contribution 
$\sum_{\beta\in P(\mathcal{C})}\mathrm{sign}(\mathcal{C})
(-1)^{|\beta|+|\gamma|}$ can be reduced to the sum of local changes of paths
corresponding to an element in $\mathcal{I}(\mathcal{C})$, which is zero.
This completes the proof of the theorem.
\end{proof}

\section{Kazhdan--Lusztig polynomials 
\texorpdfstring{$P_{\alpha,\beta}^\pm$}{P}}
\label{sec-KL2}
\subsection{Module \texorpdfstring{$\mathcal{M}^-$}{M-}}
\label{sec4-1}
\subsubsection{Case A}
Let $\alpha=\alpha_1\ldots\alpha_N\in\mathcal{P}^N$ be a binary string 
of length $N$. 
We make a pair between adjacent $2$ and $1$ (in this order) in the 
string $\alpha$ and remove it from $\alpha$.
We repeat this procedure until it becomes a sequence $1\ldots12\ldots2$. 
We call these remaining $1$'s (resp. $2$'s) as unpaired $1$'s (resp. $2$'s). 
The $(2i-1)$-th (resp. $2i$-th) unpaired 2 from right is called as 
o-unpaired (resp. e-unpaired) 2.

For simplicity, we introduce an graphical notation for these pairs, 
unpaired 1s, e- and o- unpaired 2. 
Consider a line with $N$ points. 
If $\alpha_i$ and $\alpha_j$ form a pair, then we connect $i$ and $j$ 
via an arc.
If $\alpha_i$ is an unpaired 1, we put a vertical line with a circled $1$.
If $\alpha_i$ is an e-unpaired (resp. o-unpaired) 2, we put a vertical 
line with a mark e (resp. o). 
We call this graphical notation as {\it link pattern} for Case A.

\begin{example}
Let $\alpha=1221222112$. The link pattern is 
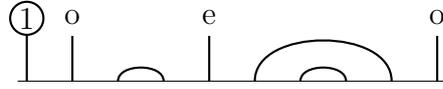
\begin{center}
\begin{tikzpicture}[scale=0.6]
\linkpattern{10}{3/4,6/9,7/8}{2/o,5/e,10/o}{1/1}{}
\end{tikzpicture}
\end{center}

\end{example}

Recall that the module $\mathcal{M}^-$ is spanned by the set of basis 
$\{m_\alpha\}_{\alpha\in\mathcal{P}_N}$. 
The space is isomorphic to $V^{\otimes N}$ where $V\cong\mathbb{C}^2$ has the 
standard basis $\{|1\rangle,|2\rangle\}$. 
When $i$-th component of the tensor product is $x\in\{1,2\}$, we denote 
it by $|x\rangle_{i}$. 
We simply write $|xx'\rangle_{ij}$ for the tensor product 
$|x\rangle_i\otimes|x'\rangle_j$ and sometimes denoted by $|xx'\rangle$ 
if the components are obvious.
Hereafter, we identify a base $m_{\alpha}, \alpha\in\{1,2\}^N$ with 
$|\alpha_1\ldots\alpha_N\rangle$. 

An arc, vertical line with e,o and a circled $1$ are building blocks
of a link pattern corresponding to a string $\alpha\in\{1,2\}^N$.
We introduce a map $\varpi^A$ from these building blocks (equivalently
a partial binary string of length $1$ and $2$) to a vector 
in $V^{\otimes 2}$ or $V$:
\begin{eqnarray*}
\tikzpic{ \linkpattern{2}{1/2}{}{}{} }
&\mapsto&|21\rangle+t^{-1}|12\rangle, \\
\tikzpic{[scale=0.5] \linkpattern{1}{}{1/o}{}{} }
&\mapsto&|2\rangle+t_N^{-1}|1\rangle, \\
\tikzpic{[scale=0.5] \linkpattern{1}{}{1/e}{}{} }
&\mapsto&|2\rangle+t^{-1}t_N|1\rangle, \\
\tikzpic{[scale=0.5]\linkpattern{1}{}{}{1/1}{}}
&\mapsto&|1\rangle
\end{eqnarray*}
Then, we extend the map $\varpi^A$ to a link pattern for a string $\alpha$ 
since a link pattern for $\alpha$ is regarded as a tensor product of 
the building blocks.

\begin{example}
\begin{eqnarray*}
\varpi^A(1212)&=&
\tikzpic{ [scale=0.5]\linkpattern{4}{2/3}{4/o}{1/1}{} } \\
&=&|1\rangle_1\otimes(|21\rangle_{23}+t^{-1}|12\rangle_{23})
\otimes(|2\rangle_{4}+t_N^{-1}|1\rangle_{4}) \\
&=&m_{1212}+t^{-1}m_{1122}+t_N^{-1}m_{1211}+t^{-1}t_N^{-1}m_{1121}
\end{eqnarray*}
\end{example}

\begin{remark}
The coefficients of $m_{\alpha}$ in $\varpi^A(\beta)$ is nothing but
the generating function of ballot strips (up to the normalization 
constant $\mathbf{t}^{|\beta|-|\alpha|}$) where the region
$\lambda^-(\beta)/\lambda^-(\alpha)$ is filled with ballot strips via 
Rule IIa.
This is because an arc corresponds to a ballot strip of length $(l,0)$ 
and an e-unpaired (resp. o-unpaired) 2 corresponds to a ballot strip 
of length $(l,2m)$ (resp. $(l,2m+1)$). 
\label{remarkA}
\end{remark}

Recall that $w^-(\alpha)$ is a minimal length representative in the 
coset.
Let us fix a reduced word $w^-(\alpha)$ and denote the ordered product
by $\stackrel{\longleftarrow}{\prod}_{(i,j)\in\lambda^-(\alpha)}s_i$.

\begin{lemma}
An element $\varpi^A(\alpha), \alpha\in\mathcal{P}_N$ is factorized as
\begin{eqnarray*}
\varpi^A(\alpha)
=\prod^{\longleftarrow}_{(i,j)\in\lambda^-(\alpha)}
(T_i+\mathbf{t}^{-l(s_i)})m_{1\ldots1}.
\end{eqnarray*}
\label{lemmafactor1}
\end{lemma}
\begin{proof}
We prove the statement by induction.
From the definition of the map $\varpi^A$, we have 
$\varpi^A(1\ldots1)=m_{1\ldots1}$ and 
\begin{eqnarray*}
\varpi^A(1\ldots12)&=&
\tikzpic{[scale=0.5]\linkpattern{4}{}{4/o}{1/1,3/1}{} 
\draw(2,0.5)node{\ldots};
} \\
&=&(T_N+t_N^{-1})m_{1\ldots1}.
\end{eqnarray*}

Fix $\beta\in\mathcal{P}_N$. 
We assume that the statement holds true for $\varpi^A(\alpha)$ for all 
$\alpha<\beta$. 
Then, there exists $\alpha<\beta$ and an integer $i$ such that 
$\beta=s_i.\alpha$ with $1\le i\le N$. 
Since $\varpi^A(\alpha)$ is a tensor product of the building blocks, 
it is enough to check the action of $T_i+\mathbf{t}^{-l(s_i)}$ on 
a local part of $\varpi^A(\alpha)$ involving $\alpha_i$ and $\alpha_{i+1}$.
\begin{enumerate}[(i)]
\item In case of $1\le i\le N-1$, we have 
$(\alpha_i,\alpha_{i+1})=(1,2)$. 
We have four cases: 
\begin{eqnarray*}
(T_i+t^{-1})\tikzpic{[scale=0.5]\linkpattern{2}{}{2/x}{1/1}{}}
&=&\tikzpic{[scale=0.5]\linkpattern{2}{1/2}{}{}{} },
\quad \mathrm{x}=\mathrm{e,o}\\
(T_i+t^{-1})
\tikzpic{[scale=0.5]
\linkpattern{4}{2/4}{}{1/1}{}
\draw(1,0)node[anchor=north]{\rm i}(2,0)node[anchor=north]{\rm i+1};
}
&=&\tikzpic{[scale=0.5]
\linkpattern{4}{1/2}{}{4/1}{} 
\draw(1,0)node[anchor=north]{\rm i}(2,0)node[anchor=north]{\rm i+1};
} \\
(T_i+t^{-1})
\tikzpic{[scale=0.5]
\linkpattern{4}{1/3}{4/x}{}{}
\draw(3,0)node[anchor=north]{\rm i}(4,0)node[anchor=north]{\rm i+1};
}
&=&\tikzpic{[scale=0.5]
\linkpattern{4}{3/4}{1/x}{}{} 
\draw(3,0)node[anchor=north]{\rm i}(4,0)node[anchor=north]{\rm i+1};
}
\quad \mathrm{x}=\mathrm{e,o}, \\
(T_i+t^{-1})
\tikzpic{[scale=0.5]
\linkpattern{4}{1/2,3/4}{}{}{}
\draw(2,0)node[anchor=north]{\rm i}(3,0)node[anchor=north]{\rm i+1};
}
&=&\tikzpic{[scale=0.5]
\linkpattern{4}{1/4,2/3}{}{}{} 
\draw(2,0)node[anchor=north]{\rm i}(3,0)node[anchor=north]{\rm i+1};
}.
\end{eqnarray*}
Now it is clear that right hand sides of above equation indicate 
$(\beta_i,\beta_{i+1})=(2,1)$ and all coefficients are one.

\item In case of $i=N$, we have $\alpha_N=1$.
We have two cases:
\begin{eqnarray*}
(T_N+t_N^{-1})\tikzpic{[scale=0.5]
\linkpattern{2}{1/2}{}{}{}
\draw(2,0)node[anchor=north]{\rm N};  }
&=&
\tikzpic{[scale=0.5]
\linkpattern{2.5}{}{1/e,2.5/o}{}{}
\draw(2,0)node[anchor=north]{\rm N};
}, \\
(T_N+t_N^{-1})\tikzpic{[scale=0.5]
\linkpattern{1}{}{}{1/1}{}
}
&=&\tikzpic{[scale=0.5]
\linkpattern{1}{}{1/o}{}{}
}
\end{eqnarray*}
As expected, $\beta_N=2$ and the coefficients are one.
\end{enumerate}
In both cases, the obtained expressions are nothing 
but $\varpi^A(\beta)$.
\end{proof}

We describe the action of $T_i+\mathbf{t}^{-l(s_i)}$ on $\varpi^A(\beta)$.
This is reduced to a local action of $T_i+\mathbf{t}^{-l(s_i)}$ on a 
partial binary strings.
Together with the proof of Lemma \ref{lemmafactor1}, remaining non-trivial
cases are as follows.
\begin{eqnarray*}
(T_i+t^{-1})\varpi^A(21)
&=&[2]\varpi^A(21), \\
(T_i+t^{-1})\tikzpic{[scale=0.5]
\linkpattern{2}{}{1/e,2/o}{}{}
}
&=&(tt_N^{-1}+t^{-1}t_N)\varpi^A(21), \\
(T_i+t^{-1})\tikzpic{[scale=0.5]
\linkpattern{2}{}{1/o,2/e}{}{}
}
&=&(t_N^{-1}+t_N)\varpi^A(21), \\
(T_i+t^{-1})\tikzpic{[scale=0.5]
\linkpattern{3}{2/3}{1/x}{}{}
\draw(1,0)node[anchor=north]{\rm i}(2,0)node[anchor=north]{\rm i+1};
}
&=&\tikzpic{[scale=0.5]
\linkpattern{3}{1/2}{3/x}{}{}
\draw(1,0)node[anchor=north]{\rm i}(2,0)node[anchor=north]{\rm i+1};
}, \quad \mathrm{x}=\mathrm{e,o}.
\end{eqnarray*}

\begin{theorem}
$\varpi^A(\alpha), \alpha\in\mathcal{P}_N$ is the Kazhdan--Lusztig 
basis $C_{\alpha}^{A,-}$.
\label{KLA}
\end{theorem}
\begin{proof}
From Lemma \ref{lemmafactor1}, an element $\varpi^A(\beta)$ is invariant
under the involution, as $T_i+\mathbf{t}^{-l(s_i)}$ is invariant.
From the definition of $\varpi^A(\beta)$, it is clear that the 
coefficient of $m_{\alpha}$ is in $\mathbb{Z}(\Gamma_{-}^{A})$  
and a monomial.
Further, the degree is less than or equal to 
$t^{l'(\alpha)-l'(\beta)}$ (as a polynomial in $t^{-1}$).
This completes the proof of the theorem.
\end{proof}

Recall that we associate a link pattern with a path $\beta$, which is a 
set of pairs between $2$ and $1$, o- and e-unpaired $2$. 
Define a set of paths by $F^A(\beta)$ as
\begin{eqnarray*}
F^A(\beta):=\{\alpha\le\beta: 
\mathrm{Some\ pairs\ and\ unpaired\  2s\ are\ flipped}
\}
\end{eqnarray*}
where by flipped we mean switching $2$ and $1$ in the pair of the binary
string $\beta$ and changing from $2$ to $1$ in an unpaired $2$ of $\beta$.
For a path $\alpha\in F^A(\beta)$, define integers
\begin{eqnarray*}
d^A(\alpha,\beta)
&=&\{\text{the number of flipped pairs}\}, \\
d^A_e(\alpha,\beta)
&=&\{\text{the number of flipped e-unpairs}\}, \\
d^A_o(\alpha,\beta)
&=&\{\text{the number of flipped o-unpairs}\}. 
\end{eqnarray*}
From Remark \ref{remarkA} and Theorem \ref{KLA}, we have 
\begin{corollary}
The generating functions $Q_{\alpha,\beta}^{A,II,-}$ and 
$P_{\alpha,\beta}^{-}$ for Case A is equal:
\begin{eqnarray*}
\mathbf{t}^{|\alpha|-|\beta|}Q_{\alpha,\beta}^{A,II,-}
&=&\mathbf{t}^{|\alpha|-|\beta|}P_{\alpha,\beta}^{A,-} \\
&=&t^{-d^A(\alpha,\beta)}t_N^{-d_o^A(\alpha,\beta)}
(t/t_N)^{-d_e^A(\alpha,\beta)}.
\end{eqnarray*}
\end{corollary}

\subsubsection{Case B}
Let $\alpha\in\mathcal{P}_N$. 
For the graphical notation, we make pairs between $2$'s and $1$'s.
Then, we have remaining unpaired $1$'s and $2$'s as Case A.
If $\alpha_i$ is the $j$-th ($1\le j\le m$) unpaired $2$ from 
right, we put a vertical line with the integer $m+1-j$. 
If $\alpha_i$ and $\alpha_{i'}$ with $i<i'$ are the $j$-th and 
$(j+1)$-th unpaired $2$'s with $j\ge m+1$ and $j-m+1\in2\mathbb{Z}$, 
we put vertical lines (on the $i$-th and $i'$-th point) whose 
endpoints are connected by a dotted line.
If $\alpha_i$ is an unpaired $1$ or a remaining unpaired $2$ not 
classified above, then we put a vertical line with a circled $1$ or 
a circled $2$ respectively on the $i$-th point.
Note that the number of unpaired $2$'s with a circled $2$ is at most one.
We call this graph as a {\it link pattern} for Case B.

\begin{example}
Let $\alpha=122212222112$ and $m=2$. 
The link pattern is 
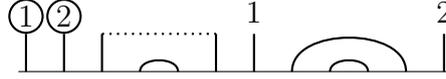
\begin{center}
\tikzpic{[scale=0.5]
\linkpattern{12}{4/5,8/11,9/10}{7/1,12/2}{1/1,2/2}{3/6}
}
\end{center}
\end{example}

We define the map $\varpi^B$ from the building blocks to a vector
in $V$ or $V^2$:
\begin{eqnarray*}
\tikzpic{[scale=0.5]  \linkpattern{2}{1/2}{}{}{} }
&\mapsto&|21\rangle+t^{-1}|12\rangle, \\
\tikzpic{[scale=0.5]  \linkpattern{1}{}{1/p}{}{} }
&\mapsto&|2\rangle+(-1)^{m-p}t^{-p}|1\rangle, \qquad 1\le p\le m, \\
\tikzpic{[scale=0.5]  \linkpattern{2}{}{}{}{1/2} }
&\mapsto&|22\rangle+t^{-1}|11\rangle, \\
\tikzpic{[scale=0.5]  \linkpattern{1}{}{}{1/x}{} }
&\mapsto&|x\rangle, \quad x\in\{1,2\}.
\end{eqnarray*}
Together with the map from a binary string to a link pattern, we 
naturally extend the map $\varpi^B$ from a binary string to 
a vector in $\mathcal{M}^-$, and denote it by $\varpi^B$.

\begin{remark}
The coefficients of $m_{\alpha}$ in $\varpi^B(\beta)$ is nothing but
the generating function of ballot strips (up to the normalization 
constant $\mathbf{t}^{|\beta|-|\alpha|}$) where the region 
$\lambda^-(\beta)/\lambda^-(\alpha)$ is filled with ballot strips
via Rule IIa and IIb.
\label{remarkB}
\end{remark}

Unlike Case A, there is no factorization property for $\varpi^B(\alpha)$.
However, we have
\begin{theorem}
An element $\varpi^B(\alpha), \alpha\in\mathcal{P}_N$, is the 
Kazhdan--Lusztig basis $C_{\alpha}^-$ for Case B.
\label{KLB}
\end{theorem}
\begin{proof}
From the definition of the map $\varpi^B$, it is clear that the coefficient
of $m_{\alpha}$ in $\varpi^B(\beta)$ is $1$ for $\alpha=\beta$ and those 
of $m_{\alpha}$ for $\alpha<\beta$ is in $t^{-1}\mathbb{Z}[t^{-1}]$. 
Further, the degree is less than or equal to 
$\mathbf{t}^{|\alpha|-|\beta|}$ (as a polynomial in $t^{-1}$). 
Therefore, it is enough to prove that $\varpi^B(\beta)$ is invariant under 
the bar involution.
We prove it by induction. The first two elements $\varpi^B(1\ldots1)$
and $\varpi^B(1\ldots12)$ are invariant since 
$\overline{m_{1\ldots1}}=m_{1\ldots1}$, 
$\varpi^B(1\ldots12)=(T_N+t_N^{-1})m_{1\ldots1}$ and 
$\overline{T_N+t_N^{-1}}=T_N+t_N^{-1}$.

Fix $\beta\in\mathcal{P}_N$. 
We assume that $\varpi^B(\alpha)$ for all $\alpha<\beta$ are invariant
under the bar involution.
Then, there exists $\alpha$ ($\alpha<\beta$) and the integer 
$1\le i\le N$ such that $\beta=s_i.\alpha$ where $s_i\in S_N^C$. 
Since $\varpi^B(\beta)$ is a tensor product of building blocks, 
it is enough to check the local action of $T_i$ on a partial string
and the local invariance under the bar involution.
\begin{enumerate}[(i)]
\item In case of $1\le i\le N-1$, {\it i.e.}, 
$(\alpha_i,\alpha_{i+1})=(1,2)$. 
The local actions of $T_i+t^{-1}$ on a partial binary string are
\begin{eqnarray*}
(T_i+t^{-1})\tikzpic{[scale=0.5] \linkpattern{2}{}{2/p}{1/1}{} }
&=&\tikzpic{[scale=0.5] \linkpattern{2}{1/2}{}{}{} }, 
\qquad 1\le p\le m, \\
(T_i+t^{-1})\tikzpic{[scale=0.5] 
\linkpattern{3.5}{2/3.5}{}{1/1}{} 
\draw(1,0)node[anchor=north]{\rm i}(2,0)node[anchor=north]{\rm i+1};
}
&=&\tikzpic{[scale=0.5] \linkpattern{3.5}{1/2}{}{3.5/1}{} 
\draw(1,0)node[anchor=north]{\rm i}(2,0)node[anchor=north]{\rm i+1};} \\
(T_i+t^{-1})\tikzpic{[scale=0.5] 
\linkpattern{3.5}{1/2.5}{3.5/p}{}{} 
\draw(2.5,0)node[anchor=north]{\rm i}(3.5,0)node[anchor=north]{\rm i+1};
}
&=&\tikzpic{[scale=0.5] \linkpattern{3.5}{2.5/3.5}{1/p}{}{} 
\draw(2.5,0)node[anchor=north]{\rm i}(3.5,0)node[anchor=north]{\rm i+1};
},\qquad 1\le p\le m, \\
(T_i+t^{-1})\tikzpic{[scale=0.5] 
\linkpattern{3.5}{}{}{1/1}{2/3.5} 
\draw(1,0)node[anchor=north]{\rm i}(2,0)node[anchor=north]{\rm i+1};
}
&=&\tikzpic{[scale=0.5] \linkpattern{3.5}{1/2}{}{3.5/2}{} 
\draw(1,0)node[anchor=north]{\rm i}(2,0)node[anchor=north]{\rm i+1};} \\
(T_i+t^{-1})\tikzpic{[scale=0.5] 
\linkpattern{5}{1/2.5}{}{}{3.5/5} 
\draw(2.5,0)node[anchor=north]{\rm i}(3.5,0)node[anchor=north]{\rm i+1};
\draw(1,0)node[anchor=north]{\rm j}(5,0)node[anchor=north]{\rm k};
}
&=&\tikzpic{[scale=0.5] \linkpattern{5}{2.5/3.5}{}{}{1/5} 
\draw(2.5,0)node[anchor=north]{\rm i}(3.5,0)node[anchor=north]{\rm i+1};
\draw(1,0)node[anchor=north]{\rm j}(5,0)node[anchor=north]{\rm k};
} \\
(T_i+t^{-1})\tikzpic{[scale=0.5] 
\linkpattern{5}{1/2.5,3.5/5}{}{}{} 
\draw(2.5,0)node[anchor=north]{\rm i}(3.5,0)node[anchor=north]{\rm i+1};
}
&=&\tikzpic{[scale=0.5] \linkpattern{5}{1/5,2.5/3.5}{}{}{} 
\draw(2.5,0)node[anchor=north]{\rm i}(3.5,0)node[anchor=north]{\rm i+1};}
\end{eqnarray*}
where $j<i$ and $k<j$ or $i+1<k$ in the fifth equation.
In all cases, the right hand sides are $\varpi^B(\beta)$ and invariant
under the bar involution.

\item In case of $i=N$, {\it i.e.}, $\alpha_N=1$ and $\beta_N=2$. 
Since the image of $\varpi^B$ is a tensor product of building blocks, 
it is enough to check the action of $T_N+t^{-m}$ on the binary string
of the form $2\ldots21$ which has some unpaired $2$'s.
We have 
\begin{eqnarray}
\label{caseA-10}
\quad
\varpi^B(\underbrace{2\ldots2}_{l+1})=
(T_N+t^{-m})\varpi^B(\underbrace{2\ldots2}_{l}1)
-\sum_{k=k_0}^m(-1)^{m-k}\langle k-1\rangle
\varpi^B(\underbrace{2\ldots2}_{l-m+k-1}1\underbrace{2\ldots2}_{m-k+1})
\label{eqnind}
\end{eqnarray}
where $k_0=\max(1,m-l+1)$. 
For the proof of Eqn.(\ref{eqnind}), it is enough to check the following 
three cases by using induction.
\begin{eqnarray*}
(T_N+t^{-m})\tikzpic{[scale=0.5]
\linkpattern{5.5}{4.5/5.5}{1/l,2/l+1,3.5/m}{}{}
\draw(2.75,0.5)node{\ldots}; }
&=&\tikzpic{[scale=0.5] 
\linkpattern{3.5}{}{1/l-2,2/l-1,3.5/m}{}{} 
\draw(2.75,0.5)node{\ldots};}
+(-1)^{m-l}\langle l-1\rangle
\tikzpic{[scale=0.5]
\linkpattern{5.5}{1/2}{3/l,4/l+1,5.5/m}{}{}
\draw(4.75,0.5)node{\ldots}; } \\ 
&&+\sum_{k=l+1}^m(-1)^{m-k}\langle k-1\rangle
\tikzpic{[scale=0.5]
\linkpattern{7}{3.5/4.5}{1/l,2.5/k-1,5.5/k,7/m}{}{}
\draw(3.5/2,0.5)node{\ldots}(12.5/2,0.5)node{\ldots}; } \\
&&+
(-1)^{m-l+1}\langle l-2\rangle 
\tikzpic{[scale=0.5] 
\linkpattern{4.5}{}{2/l-1,3/l,4.5/m}{1/1}{} 
\draw(7.5/2,0.5)node{\ldots}; 
},
\qquad 3\le l\le m, \\
(T_N+t^{-m})\tikzpic{[scale=0.5]
\linkpattern{5.5}{4.5/5.5}{1/1,2/2,3.5/m}{}{}
\draw(2.75,0.5)node{\ldots}; }
&=&
\tikzpic{[scale=0.5]
\linkpattern{5.5}{}{3/1,4/2,5.5/m}{}{1/2}
\draw(9.5/2,0.5)node{\ldots};  }
+(-1)^{m-1}
\tikzpic{[scale=0.5]
\linkpattern{5.5}{1/2}{3/1,4/2,5.5/m}{}{}
\draw(9.5/2,0.5)node{\ldots}; } \\
&&+\sum_{k=2}^{m}(-1)^{m-k}\langle k-1\rangle
\tikzpic{[scale=0.5]
\linkpattern{7}{3.5/4.5}{1/1,2.5/k-1,5.5/k,7/m}{}{}
\draw(3.5/2,0.5)node{\ldots}(12.5/2,0.5)node{\ldots}; } \\
(T_N+t^{-m})\tikzpic{[scale=0.5]
\linkpattern{5.5}{4.5/5.5}{1/2,2/3,3.5/m}{}{}
\draw(2.75,0.5)node{\ldots}; }
&=&
\tikzpic{[scale=0.5]
\linkpattern{4.5}{}{2/1,3/2,4.5/m}{1/2}{}
\draw(7.5/2,0.5)node{\ldots}; }
+(-1)^{m-1}\tikzpic{[scale=0.5]
\linkpattern{4.5}{}{2/1,3/2,4.5/m}{1/1}{}
\draw(7.5/2,0.5)node{\ldots}; }  \\
&&+\sum_{k=2}^{m}(-1)^{m-k}\langle k-1\rangle
\tikzpic{[scale=0.5]
\linkpattern{7}{3.5/4.5}{1/2,2.5/k-1,5.5/k,7/m}{}{}
\draw(3.5/2,0.5)node{\ldots}(12.5/2,0.5)node{\ldots}; }
\end{eqnarray*}

Note that the binary string $\underbrace{2\ldots2}_{l+1}$ is 
bigger (in the Bruhat order)
than any other binary strings in the right hand side of 
Eqn.(\ref{eqnind}). 
From the assumption of the induction together with the fact that 
$(T_N+t^{-m})$ and $\langle k-1\rangle$ are also invariant, 
the left hand side of Eqn.(\ref{caseA-10}) is invariant. 
Therefore, $\varpi^B(2\ldots2)$ is invariant. 
\end{enumerate}

In both cases, the image of $\varpi^B$ is invariant under the bar 
involution. 
%%Note also that the coefficient of $m_\alpha$ in $\varpi^B(\beta)$
%%is in $t^{-1}\mathbb{Z}[t^{-1}]$ when $\alpha\neq\beta$ and $1$ 
%%when $\alpha=\beta$. 
This completes the proof that $\varpi^B(\beta)$ is Kazhdan--Lusztig
basis $C_{\alpha,\beta}^-$.
\end{proof}

Define a set of paths by $F^B(\beta)$ as 
\begin{eqnarray*}
F^B(\beta):=\{\alpha\le\beta: 
\mathrm{Some\ pairs\ and\ unpaired\ 2s\ are\ flipped}\},
\end{eqnarray*}
where by flipped we mean switching $2$ and $1$ in the pair of the 
binary string $\beta$, changing $2$ to $1$ in an unpaired $2$ of 
$\beta$ and changing two $2$'s to $1$'s (at the same time) in a 
paired $2$'s. 
For a path $\alpha\in F^B(\beta)$, define integers
\begin{eqnarray*}
d^B(\alpha,\beta)&:=&
\{\mathrm{the\ number\ of\ flipped\ 21\ pairs\ and\ 22\ pairs}\}, \\
d^B_{p}(\alpha,\beta)&:=&
\{\mathrm{the\ number\ of\ flipped\ unpair\ 2\ with\ integer\ }p\}. 
\end{eqnarray*}
Note that $d^B_p(\alpha,\beta)$ is either $0$ or $1$. 
As a consequence of Theorem \ref{KLB} with Remark \ref{remarkB}, 
we have 

\begin{corollary}
A Kazhdan--Lusztig polynomial $P_{\alpha,\beta}^{B,-}$ is a monomial of 
$t^{-1}$, and equal to the generating function 
$Q_{\alpha,\beta}^{B,II,-}$
\begin{eqnarray*}
\mathbf{t}^{|\alpha|-|\beta|}Q_{\alpha,\beta}^{B,II,-}
&=&\mathbf{t}^{|\alpha|-|\beta|}P_{\alpha,\beta}^{B,-}  \\
&=&(-1)^{\sigma'}t^{-d},
\end{eqnarray*}
where the degrees $\sigma'$ and $d$ are given by
\begin{eqnarray*}
\sigma'&=&\sum_{p=1}^{m}(m-p)d_p^B(\alpha,\beta),  \\
d&=&d^B(\alpha,\beta)+\sum_{1\le p\le m}p d_p^B(\alpha,\beta).
\end{eqnarray*}
\end{corollary}

\subsection{Module \texorpdfstring{$\mathcal{M}^+$}{M+}}
We prove that the generating functions 
$Q_{\alpha,\beta}^{X,II,-}, X=A,B$, are equal to the Kazhdan--Lusztig 
polynomials $P_{\alpha,\beta}^-$. 
The generating functions $Q_{\alpha,\beta}^\pm$ satisfy the inversion 
relation (Theorem \ref{invBallot}) which is exactly the same as the 
inversion formula (Theorem \ref{invKL}).
Therefore, we have the following:
\begin{theorem}
The generating functions $Q_{\alpha,\beta}^{A,I,+}$ (resp. 
$Q_{\alpha,\beta}^{B,I,+}$) is equal to the Kazhdan--Lusztig polynomials
$P^+_{\alpha,\beta}$ for Case A (resp. Case B).
\label{thrmgfKL}
\end{theorem}

\subsubsection{Factorization for Case A}
For each binary string $\alpha$, we define a set of integers 
$\{r_{i,j}: (i,j)\in S^+(\alpha)\}$ (recall the definition in 
Eqn.(\ref{defS})). 
They are defined recursively by 
\begin{eqnarray*}
r_{i,j}:=\left\{
\begin{array}{cc}
\max(r_{i-1,j-1},r_{i+1,j-1})+1, & (i,j)\in S^+(\alpha), \\
0, & \mathrm{otherwise}. 
\end{array}\right.
\end{eqnarray*}
We define a factorized element 
$\widetilde{C_{\alpha}}, \alpha\in\mathcal{P}_N$ on $\mathcal{M}^+$,
\begin{eqnarray*}
\widetilde{C_{\alpha}}:=
\prod_{(i,j)\in\lambda^+(\alpha)}^{\longleftarrow}
T_i(r_{i,j})m_{1\ldots1},
\end{eqnarray*}
where 
\begin{eqnarray*}
T_i(p):=\left\{
\begin{array}{cc}
\displaystyle T_i+\frac{t^{-p}}{[p]}, & 1\le i\le N-1, \\
\displaystyle T_N-t_N+\frac{[\lceil p/2\rceil]}{[p]}
(t_Nt^{\lfloor p/2\rfloor}+t_N^{-1}t^{-\lfloor p/2\rfloor}), 
& i=N. 
\end{array}
\right.
\end{eqnarray*}
Here, $\lceil\cdot\rceil$ and $\lfloor\cdot\rfloor$ is the 
ceiling and floor function.

\begin{example}
Let $\alpha=21122$. The set $r_{i,j}$ of integers is given by 
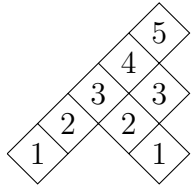
\begin{center}
\begin{tikzpicture}[scale=0.4]
\draw(0,0)--(1,-1)--(3,1)--(5,-1)--(6,0)--(5,1)--(6,2)--
(5,3)--(6,4)--(5,5)--(0,0)(1,1)--(2,0)(2,2)--(3,1)--(5,3)
(3,3)--(5,1)--(4,0)(4,4)--(5,3);
\draw(1,0)node{$1$}(2,1)node{$2$}(3,2)node{$3$}
(4,3)node{$4$}(5,4)node{$5$}
(4,1)node{$2$}(5,2)node{$3$}(5,0)node{$1$};
\end{tikzpicture}
\end{center}
The associated factorized expression is 
\begin{eqnarray*}
\widetilde{C_\alpha}
=
T_5(1)T_4(2)T_5(3)T_1(1)T_2(2)T_3(3)T_4(4)T_5(5)m_{11111}. 
\end{eqnarray*}
\end{example}

\begin{theorem}
The factorized element $\widetilde{C_\alpha}$ is the Kazhdan--Lusztig 
basis $C_\alpha^{A,-}$.
\end{theorem}
\begin{proof}
We omit the details since we can apply the same method in \cite{KL00} 
to our case.
\end{proof}

\begin{remark}
The factorized element $\widetilde{C_\alpha}$ appeared in the study of 
the quantum Knizhnik--Zamolodchikov equation in \cite{dGP}. 
This is a natural generalization of the factorization 
obtained in \cite{KL00}.
\end{remark}

\section{Binary tree}
\label{sec-binary}
\subsection{Notations}
Following \cite{Boe,LS81}, we introduce some terminologies to describe
binary trees for both Case A and B.

Let $\mathcal{Z}$ be a set of binary strings such that 
$\emptyset\in\mathcal{Z}$, $z\in\mathcal{Z}\Rightarrow1z2\in\mathcal{Z}$
and if $z_1,z_2\in\mathcal{Z}$ then the concatenation 
$z_1z_2\in\mathcal{Z}$. 
A binary string $\alpha\in\mathcal{P}_N$ is of the form
\begin{eqnarray*}
\alpha=\underline{2}z_1\underline{2}z_2\ldots\underline{2}z_p\underline{1}
z_{p+1}\underline{1}z_{p+2}\ldots\underline{1}z_q 
\end{eqnarray*}
for some integer 
$p,q\ge0$ with $z_i\in\mathcal{Z}$. 
We call an underlined $1$ (resp. $2$) as an unpaired $1$ (resp. 
unpaired $2$). 

We denote by $||\alpha||$ the length of a binary string $\alpha$ and 
by $||\alpha||_{\sigma}$ the number of $\sigma$ in the string $\alpha$.
Let $\alpha,\beta\in\mathcal{P}_N$ with $\alpha\le\beta$ ($\epsilon=+$) 
and $\alpha=\alpha'vw\alpha'', \beta=\beta'\underline{12}\beta''$ with 
$||\alpha'||=||\beta'||$ and $v,w\in\{1,2\}$. 
A {\it capacity} of the edge corresponding to the underlined $1$ and $2$
in $\beta$ is defined by 
\begin{eqnarray*}
\mathrm{cap}(12):=||\alpha'v||_1-||\beta'1||_1.
\end{eqnarray*}
Similarly, if $\alpha=\alpha'v$ and $\beta=\beta'\underline{1}$ with 
$v\in\{1,2\}$, then the capacity of the edge corresponding to the 
underlined $1$ is 
\begin{eqnarray*}
\mathrm{cap}(1):=||\alpha||_1-||\beta||_1.
\end{eqnarray*}
Note that the condition $\alpha\le\beta$ implies a capacity is always
non-negative.

The capacity of $\beta$ with respect to $\alpha$ is the collection of 
capacities of pairs of adjacent $1$ and $2$ in $\alpha$ and that of 
the rightmost $1$ in $\beta$ if it exists.

We associate a binary tree $A(\alpha)$ with $\alpha\in\mathcal{P}_N$
in the following subsections for Case A and B.

\subsubsection{Case A}
We divide unpaired $1$'s into two classes. 
The $(2i-1)$-th (resp. $2i$-th) unpaired $1$ from right is called
o-unpaired (resp. e-unpaired) $1$.

A binary tree $A(\alpha)$ satisfies
\begin{enumerate}[($\diamondsuit$1)]
\item $A(\emptyset)$ is the empty tree.
\item $A(2w)=A(w)$.
\item $A(zw)$, $z\in\mathcal{Z}$ is obtained by attaching the tree 
for $A(z)$ and $A(w)$ at their roots.
\item $A(1z2)$, $z\in\mathcal{Z}$ is obtained by attaching an edge 
just above the tree $A(z)$.
\item If unpaired $1$ in $\underline{1}w$ is e-unpaired (resp. 
o-unpaired) $1$, $A(1w)$ is obtained by attaching an edge just 
above the tree $A(w)$ and mark the edge with ``e" (resp. ``o").
\end{enumerate}

We write the capacities of $\beta$ with respect to $\alpha$ as 
integers on leaves of the binary tree $A(\beta)$ (See 
Example \ref{exbt}). 
Denote by $A(\beta/\alpha)$ a tree $A(\beta)$ equipped with capacities
with respect to $\alpha$.
A {\it labelling} of $A(\beta/\alpha)$ is a set 
of non-negative integers on edges of $A(\beta)$ satisfying
\begin{enumerate}[($\clubsuit$1)]
\item An integer on a edge connecting to a leaf is less than or equal 
to its capacity.
\item Integers on edges are non-increasing from leaves to the root.
\end{enumerate} 

Let $\sigma$ be the sum of the labels on edges without marks 
``e" and ``o", and $\sigma_e$ (resp. $\sigma_o$) be the sum of 
the labels on edges with a mark ``e" (resp. ``o").  
\begin{defn}
The generating function of labellings on the tree $A(w/v)$:
\begin{eqnarray*}
R_{v,w}^A(t^2,t_N^2):=
\sum_{\nu}t^{2\sigma}(-t_N^2)^{\sigma_o}(-t^2/t_N^2)^{\sigma_e},
\end{eqnarray*}
where the sum runs over all possible labellings $\nu$ of $A(w/v)$. 
\label{defngfbtA}
\end{defn}

\begin{example}
Let $(\alpha,\beta)=(1111111,2211211)$. 
The binary tree $A(\beta)$ and a labelling is
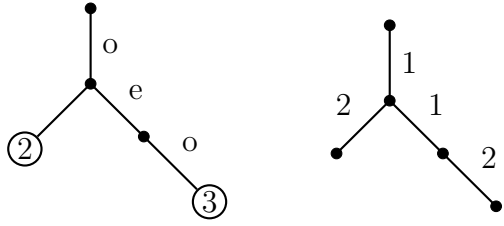
\begin{center}
\tikzpic{
\draw[thick](0,0)--node[anchor=west]{\rm o}(0,1)
(0,0)--(-0.7,-0.7)node[circle,inner sep=1pt,draw,anchor=north east]{$2$}
(0,0)--node[anchor=south west]{\rm e}(0.7,-.7)
--node[anchor=south west]{\rm o}
(1.4,-1.4)node[circle,inner sep=1pt,draw,anchor=north west]{$3$};
\filldraw(0,0)circle(2pt)(0,1)circle(2pt)
(0.7,-0.70)circle(2pt);
}\qquad
\tikzpic{
\draw[thick](0,0)--node[anchor=west]{\rm 1}(0,1)
(0,0)--node[anchor=south east]{\rm 2}(-0.7,-0.7)
(0,0)--node[anchor=south west]{\rm 1}(0.7,-.7)
--node[anchor=south west]{\rm 2}(1.4,-1.4);
\filldraw(0,0)circle(2pt)(0,1)circle(2pt)
(0.7,-0.70)circle(2pt)(-.7,-.7)circle(2pt)(1.4,-1.4)circle(2pt);
}

\end{center}
The capacities of a pair $12$ and o-unpaired $2$ are $2$ and $3$
respectively.
The weight of the labelling is $t^4t^4_N$. 
\label{exbt}
\end{example}

The generating functions defined in Section \ref{sec-comb} are related
to the generating functions $R_{v,w}^A$ as follows.
\begin{theorem}
We have 
\begin{eqnarray*}
Q_{\alpha,\beta}^{A,I,+}=R_{\alpha,\beta}^A.
\end{eqnarray*}
\label{thrmKLA}
\end{theorem}
The proof of Theorem will be given in Section \ref{sec5-2}.

\subsubsection{Case B}
If $\alpha_i$ is the $(m+1-j)$-th ($1\le j\le m$) unpaired $1$ 
from right, we call this as {\it $j$-terminal} $1$. 
If $\alpha_i$ and $\alpha_{i'}$ with $i<i'$ are the $j$-th and $(j+1)$-th 
unpaired $1$'s with $j\ge m+1$ and $j-m$ odd, we make a pair these $1$'s 
and call it a {\it $11$-pair}.
If $\alpha_i$ is an unpaired $1$ and not classified above, we call this 
as an {\it extra-unpair} $1$. 
Note that there is at most one extra-unpair $1$.

$A(\beta)$ is defined recursively by the following rules. 
The rules ($\diamondsuit$1)-($\diamondsuit$4) are the same as 
Case A. 
We replace ($\diamondsuit$5) by the following four conditions:
\begin{enumerate}[($\diamondsuit1$)]
\setcounter{enumi}{5}
\item[($\diamondsuit5'$)] If the underlined $1$ in $\underline{1}w$ 
is the $j$-terminal with $1\le j\le m$, $A(\underline{1}w)$ is 
obtained by putting an edge just above the tree $A(w)$.
Then mark this edge with a plus ``$+$" only when $j=1$.
\item Suppose underlined $1$ in $\underline{1}z\underline{1}w$ is 
a $11$-pair. 
The tree $A(1z1w)$ is obtained by attaching an edge above the root of 
$A(zw)$.
We mark the edge with a plus ``$+$". 
\item If the underlined $1$ in $\underline{1}w$ is an extra-unpair $1$, 
we have $A(1w)=A(w)$.
\end{enumerate}

Further, we need an additional information on the tree. 
See \cite{Boe} for $m=1$ case.
Suppose $w=w'z_{m+2r}1\ldots z_11z_0$ with $z_i\in\mathcal{Z}$ 
and $r\ge0$ ($z_{m+2r}$ is non-empty and maximal).
Set $w''=1z_{m+2r-1}1\ldots z_11z_0$ such that $w=w'z_{m+2r}w''$ and 
$z_{m+2r}=x_sx_{s-1}\ldots x_1$ with $x_i\in\mathcal{Z}$.
Here all $x_i$'s can not be decomposed further into a product of 
non-empty elements in $\mathcal{Z}$. 
Then the tree $A(x_i)$ contains a unique maximal edge (the edge 
connecting to the root) corresponding to a pair $12$.
$A(w'')$ contains a unique maximal edge corresponding to a $11$-pair
or a $m$-terminal.
Observe that $A(x_i)\subseteq A(w)$, $A(w'')\subseteq A(w)$ as binary 
trees.
We say that the maximal edge of $A(x_i)$ (resp. $A(w'')$) 
{\it immediately precedes} the maximal edge of $A(x_{i+1})$ 
(resp. $A(x_1)$) for $1\le i\le s$.
\begin{enumerate}[($\diamondsuit1$)]
\setcounter{enumi}{7}
\item When an edge $e$ immediately precedes an edge $e'$ in the binary 
tree $A(w)$, we put a dotted arrow from the edge $e$ to the edge $e'$.
\end{enumerate}

A labelling of $A(w/v)$ is a set of non-negative integers on edges 
of $A(w)$ satisfying the following rules.
In addition to $(\clubsuit1)$ and $(\clubsuit2)$ (the same as Case A),
we require 
\begin{enumerate}[($\clubsuit1$)]
\setcounter{enumi}{2}
\item An integer attached to any edge with a plus ``+" must be even.
\item If the label on an edge is less than or equal to the labels on 
all ``preceding" edges, then the former must be even.
\end{enumerate}

\begin{example}
Let $\alpha=22111211$. 
The binary trees for $\alpha$ with $m=1,2$ and $3$ from left to right.
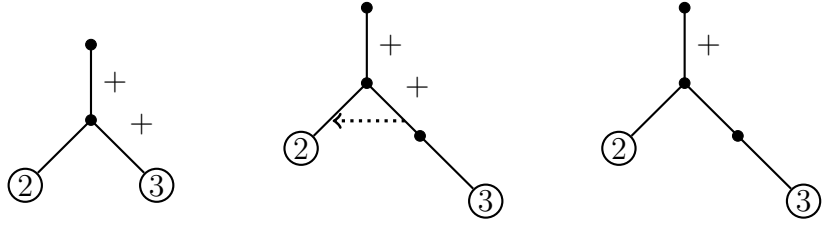
\begin{center}
\tikzpic{
\draw[thick](0,0)--node[anchor=west]{$+$}(0,1)
(0,0)--(-.7,-.7)node[circle,inner sep=1pt,draw,anchor=north east]{$2$}
(0,0)--node[anchor=south west]{$+$}
(.7,-.7)node[circle,inner sep=1pt,draw,anchor=north west]{$3$};
\filldraw(0,0)circle(2pt)(0,1)circle(2pt);
}
\qquad
\tikzpic{
\draw[thick](0,0)--node[anchor=west]{$+$}(0,1)
(0,0)--(-0.7,-0.7)node[circle,inner sep=1pt,draw,anchor=north east]{$2$}
(0,0)--node[anchor=south west]{$+$}(0.7,-.7)
--(1.4,-1.4)node[circle,inner sep=1pt,draw,anchor=north west]{$3$};
\draw[dotted,very thick,->](.5,-.5)--(-.45,-.5);
\filldraw(0,0)circle(2pt)(0,1)circle(2pt)
(0.7,-0.70)circle(2pt);
}\qquad
\tikzpic{
\draw[thick](0,0)--node[anchor=west]{$+$}(0,1)
(0,0)--(-0.7,-0.7)node[circle,inner sep=1pt,draw,anchor=north east]{$2$}
(0,0)--(0.7,-.7)
--(1.4,-1.4)node[circle,inner sep=1pt,draw,anchor=north west]{$3$};
\filldraw(0,0)circle(2pt)(0,1)circle(2pt)
(0.7,-0.70)circle(2pt);
}
\end{center}
\end{example}

Given a labelling $\nu$, let $|\nu|$ be the sum of the labels on all
edges in $A(w/v)$.
\begin{defn}
The generating function $R_{v,w}^B$ of labellings on $A(w/v)$ 
is defined by
\begin{eqnarray*}
R_{v,w}^B:=\sum_\nu t^{2|\nu|}.
\end{eqnarray*}
\label{defngfbtB}
\end{defn}

\begin{theorem}
We have 
\begin{eqnarray*}
Q_{\alpha,\beta}^{A,I,+}=R_{\alpha,\beta}^B.	
\end{eqnarray*}
\label{thrmKLB}
\end{theorem}
The proof will be given in Section~\ref{sec5-2}.

From Theorems \ref{thrmKLB}and \ref{thrmgfKL}, we have 
$P_{\alpha,\beta}^{B,+}=R_{\alpha,\beta}^B$.
From Definition \ref{defngfbtB}, the polynomials $P_{\alpha,\beta}^+$ 
have a positivity similar to the equal parameter case.

\paragraph{\bf Recurrence relation for $P_{v,w}^{X,+}$}
Let $v_1=v21v', w_1=w\underline{12}w'$ be binary strings with 
$||v_1||=||w_1||$ and $||v||=||w||$.  
Denote by $c_1$ the capacity of the underlined pair $12$ in $w_1$.
\begin{prop}
\label{prop-recurrence1}
The polynomials $P_{v,w}^{X,+}$ satisfy
\begin{eqnarray}
\label{recurrence1}
P_{v_1,w_1}^{X,+}
=
t^{2c_1}P_{vv',ww'}^{X,+}
+P_{v12v',w_1}^{X,+}, \qquad X=A,B.
\end{eqnarray}
\end{prop}
\begin{proof}
Recall that the generating function $R^{X}_{v,w}$ is the sum 
of the weight of a labelling on $A(w/v)$.
The edge $e$ connected to the leaf has the integer less than 
or equal to the capacity.
If the label on $e$ is equal to the capacity $c_1$, 
the contribution to $R_{v,w}^{X}$ is $t^{2c_1}P_{vv',ww'}^{X,+}$.
Note that the binary tree for $P_{vv',ww'}^{X,+}$ is obtained from 
the binary tree for $P_{v_1,w_1}^{X,+}$ by 
deleting the edge $e$ and the capacity of the new leaf is again $c_1$.
If the label on $e$ is less than $c_1$, a labelling is bijective to 
a labelling of the same binary tree with the capacity $c_1-1$.
The binary tree for $P_{v12v',w_1}^{X,+}$ is the same as $P_{v_1,w_1}$ but 
the capacity is $c_1-1$. 
The sum of the contribution of the two cases leads to Eqn.(\ref{recurrence1}).
\end{proof}

Similarly, let $v_2=v1, w_2=w\underline{1}$ be binary strings 
with $||v||=||w||$. 
Denote by $c_2$ the capacity of the underlined $1$ in $w_2$. 
\begin{prop}
The Kazhdan--Lusztig polynomial for Case A satisfies 
\begin{eqnarray}
\label{recurrence2}
P_{v_2,w_2}^{A,+}(t,t_N)
=(-t_N^2)^{c_2}P_{v,w}^{A,+}(t,t/t_N)
+P_{v2,w_2}^{A,+}(t,t_N).
\end{eqnarray}
\end{prop}
\begin{proof}
The underlined 1 in $w_2$ is an o-unpaired 1. 
A label on the edge $e$ connected to the leaf associated with this 
o-unpaired 1 is less than or equal to the capacity $c_1$.
If the label is equal to $c_1$, the contribution to $R_{v_2,w_2}^A$ 
is $(-t_N^2)^{c_1}P_{v,w}^{A,+}(t,t_N/t)$. 
Note that the binary tree for $P_{v,w}^{A,+}$ is obtained from the 
binary tree for $P_{v_2,w_2}^{A,+}$ by deleting the edge $e$ and 
the capacity of the new leaf is $c_1$. However, the marks ``e" and 
``o" should be exchanged in the deleted binary tree.
This exchange of ``e" and ``o" is realized in $R_{v,w}^{A}$ as 
$t_N\rightarrow t/t_N$.
If the label on $e$ is less than $c_1$, the contribution to $R_{v_2,w_2}^A$
is $P_{v2,w_2}$ by a similar argument to Proposition~\ref{prop-recurrence1}.
Adding the two contributions, we obtain Eqn.(\ref{recurrence2}).
\end{proof}

\subsection{Proofs of Theorem~\ref{thrmKLA} and Theorem~\ref{thrmKLB}}
\label{sec5-2}
To prove Theorems, we will construct a bijection between a labelling
of $A(w/v)$ and a configuration of ballot strips by Rule I introduced 
in Section~\ref{sec-comb}.
This will be done by introducing a link pattern with labelling.
Then, we will show that the generating functions are the same by counting
the power of $t$.
See also Section 4 in \cite{SZJ10}.

Let $\beta\in\mathcal{P}_N$ be a binary string. 
We denote by $\bar{\beta}$ the binary string which is obtained from 
$\beta$ by exchanging $1$ and $2$ in $\beta$.
We also denote by $\pi(\beta)$ the link pattern for $\bar{\beta}$ as 
in Section \ref{sec4-1}.
Then, the link pattern $\pi(\beta)$ is the dual graph of the binary 
tree $A(\beta)$ (see Figure~\ref{Fig4}).
In Case A, an edge without a mark (resp. with ``o" or ``e") in a binary 
tree corresponds to an arc (resp. a vertical line with ``o" or ``e") 
in the link pattern.
In Case B, an edge without ``+" in a binary tree corresponds to an arc 
(corresponding to a pair $12$) or a vertical line with the integer $p$
with $2\le p\le m$ in the link pattern.
An edge with ``+" in a binary tree corresponds to a vertical line with 
the integer $1$ or to an arc for paired $1$'s in the link pattern.
Notice that the map from link patterns to trees is not one-to-one without
fixing the string $\beta$: for some cases in Case B, we cannot distinguish
an arc from a vertical line in a link pattern by looking at only the 
binary tree.
In Case A, the map is bijective up to the ignorance of unpaired vertices 
of a link pattern.

\begin{figure}[ht]
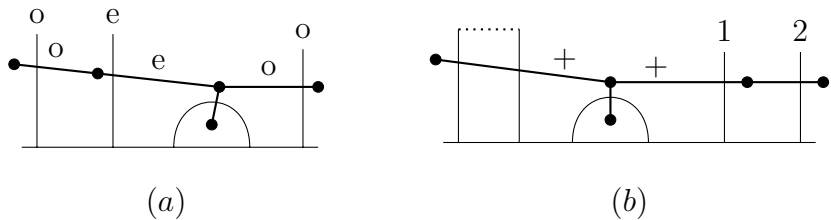

\begin{eqnarray*}
\begin{array}{ccc}
\tikzpic{
\draw(0,0)--(0,1.5)node[anchor=south]{\rm o}
(1,0)--(1,1.5)node[anchor=south]{\rm e}
(1.8,0)..controls (1.8,.8) and (2.8,.8)..(2.8,0)
(3.5,0)--(3.5,1.3)node[anchor=south]{\rm o}
(-0.2,0)--(3.7,0);
\draw[thick](-0.3,1.1)--node[anchor=south]{\rm o}(.8,0.9778)
--node[anchor=south]{\rm e}(2.4,.8)--node[anchor=south]{\rm o}(3.7,.8)
(2.4,.8)--(2.3,.3);
\filldraw(.8,.9778)circle(2pt)(-.3,1.1)circle(2pt)
(2.4,.8)circle(2pt)(3.7,.8)circle(2pt)(2.3,.3)circle(2pt);
}   & \quad &
\tikzpic{
\draw(0,0)--(0,1.5)(0.8,0)--(0.8,1.5)
(1.5,0)..controls(1.5,.8) and (2.5,.8)..(2.5,0)
(3.5,0)--(3.5,1.2)node[anchor=south]{$1$}
(4.5,0)--(4.5,1.2)node[anchor=south]{$2$}
(-0.2,0)--(4.7,0);
\draw[thick,dotted](0,1.5)--(0.8,1.5);
\draw[thick](-.3,1.1)--(2,.8)--(3.8,.8)--(4.8,.8)(2,.8)--(2,.3);
\draw(1.4,1.1)node{+}(2.6,1)node{+};
\filldraw(-.3,1.1)circle(2pt)(2,.8)circle(2pt)
(3.8,.8)circle(2pt)(4.8,.8)circle(2pt)
(2,.3)circle(2pt);
} \\
 & &  \\
(a) & & (b)
\end{array}
\end{eqnarray*}
\caption{The link pattern and the binary tree for $22111211$. 
(a) Case A. (b) Case B and $m=2$.}
\label{Fig4}
\end{figure}

We attach a labelling (a set of integers) to a link pattern. 
Fix a labelling of $A(w/v)$ and an edge $e$ with a label $n(e)$.
We put the label $n'=n(e)-n(e')$ on the corresponding pair $12$, 
paired $1$'s or a vertical line, where the edge $e'$ is the parent 
edge of $e$, unless there is no parent edge (edge connected to the root)
in which case we put $n(e)$.
See Figure~\ref{fig5} (a) and (b) for example.

A labelling of a link pattern $\pi(\beta)$ which is obtained from a 
labelling of $A(\beta/\alpha)$ satisfies the following three conditions:
1) All labels are non-negative integers. 2) Given a smallest pair $12$, 
the sum of all labels on planar pairings and paired $1$'s which surround 
the pair $12$ is less than or equal to the capacity of the pair. 
3) Given the rightmost unpaired $2$, the sum of all labels on unpaired 
$2$'s and paired $1$'s is less than or equal to the capacity of this 
unpaired $2$.

We consider a pair of paths $\alpha,\beta$ with $\alpha<\beta$ in 
$\epsilon=+$, and the associated link pattern $\pi(\beta)$ with 
a labelling. 
We associate with it a collection of ballot strips between paths $\alpha$
and $\beta$ following the three steps based on the map constructed above.

\paragraph{\bf Step 1}
In the first step, we associate a collection of ballot paths with a labelling
of a link pattern.
We stack ballot paths on top of each other forming parallel layers above 
$\beta$ in the following order.
\begin{enumerate}[\bf {1}-1]
\item Case B: 
We associate with each paired $1$'s of $\pi(\beta)$ two ballot paths (of 
the same length) which connect a half-step to the left of the left point 
of the pairing and an anchor box. 
If the pair has the label $n$, then we stack $2n$ ballot paths.
We start this process from the leftmost paired $1$'s, then move to the next 
paired $1$'s.
\item Take a vertical line with ``e" or ``o" for Case A or a vertical 
line with an integer $p, 1\le p\le m$ for Case B.
Suppose this vertical line is $n_1$-th vertical line from right end 
with the label $n_2$ and there are $n_3$ planar arcs between the 
vertical line and the right end.
Then, we stack $n_2$ ballot paths of length $(n_3,n_1)$.
In both Case A and B, we first stack ballot strips corresponding to the 
leftmost vertical line, then move right and repeat the procedure.
\item Case A and Case B: 
With each pair $12$ of $\pi(\beta)$, we associate ballot paths of length 
$(l,0)$ (that is a Dyck path) which start a half-step to the left of the 
left point of the pairing and a half-step to the right of its right point.
If the pair has a label $n$, then we stack $n$ such ballot paths on top of 
each other, forming parallel layers above $\beta$. 
We repeat the process for all pairs starting from the largest arcs and 
ending with the smallest arcs.
\end{enumerate}

When we stack ballot paths above the path $\beta$, we form parallel layers 
along the shape of $\beta$.
Note that some ballot paths may have common starting or end points in the 
Step 1, in which case they are merged into a larger ballot path.

\paragraph{\bf Step 2}
We associate the corresponding ballot strip with each ballot path obtained 
in Step 1 since a ballot strip is characterized by a ballot path.
We will show that these strips remain under the path $\alpha$ for the 
following two cases.
\begin{enumerate}[\bf 2-1]
\item Let $p$ be a smallest planar pairing, that is, connecting $i$ and 
$i+1$. 
Then the difference of heights of $\alpha$ and $\beta$ at the center of 
the pairing $p$ ({\it i.e.}, the depth of the corner in the skew Ferrers 
diagram) is by direct computation exactly the capacity of the edge $e$ 
(corresponding to $p$) in the tree $A(\beta/\alpha)$. 
In terms of $\pi(\beta)$, the number of ballot strips above this corner
is the sum of labels of pairings (pair $12$ or paired $1$'s) surround $p$.
This number is the label of $e$, which is less than or equal to the capacity.
Therefore, the ballot strips remain below $\alpha$ at this local maximum 
of $\beta$.
\item Let $v$ be the rightmost vertical line if exists.
From the construction of the bijection, the binary tree $A(\beta/\alpha)$
has the edge $e'$ (corresponding to $v$) with a capacity.
Then the difference of heights of $\alpha$ and $\beta$ at this point 
is nothing but the capacity of the edge $e'$.
In terms of $\pi(\beta)$, the number of ballot strips above this point 
is the sum of labels of unpaired $1$'s and paired $1$'s.
This number is nothing but the label of $e'$, less than or equal to 
the capacity.
Note that the capacity is equal to the number of anchor boxes in the skew
Ferrers diagram.
The ballot strips remain below $\alpha$ at the right end.
\end{enumerate}
In both cases, strips are below $\alpha$, which implies the claim of 
step 2.

\paragraph{\bf Step 3}
The last step is to fill up the remaining regions by ballot strips of length
$(0,0)$ ({\it i.e.} a single box).
From Step 1 and 2, it is clear that there is no ballot strips of length 
$(l,l')\in\mathbb{N}^2\backslash\{(0,0)\}$ on top of a single box. 
See Figure~\ref{fig5} for an example.

\begin{figure}[ht]
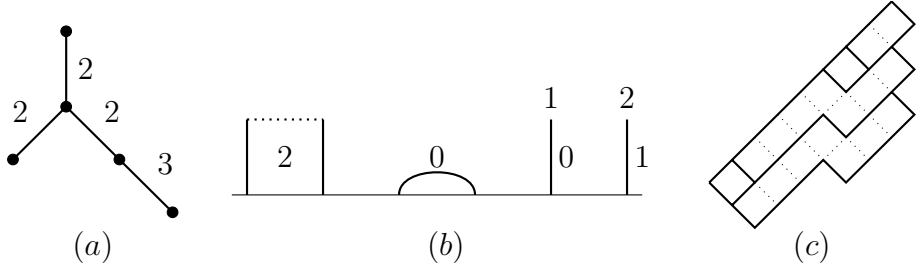

\begin{eqnarray*}
\begin{array}{ccc}
\tikzpic{
\draw[thick](0,0)--node[anchor=west]{$2$}(0,1)
(0,0)--node[anchor=south east]{$2$}(-.7,-.7)
(0,0)--node[anchor=south west]{$2$}(.7,-.7)
--node[anchor=south west]{$3$}(1.4,-1.4);
\filldraw(0,0)circle(2pt)(0,1)circle(2pt)(-.7,-.7)circle(2pt)
(0.7,-0.7)circle(2pt)(1.4,-1.4)circle(2pt);
} & 
\tikzpic{
\linkpattern{6}{3/4}{5/1,6/2}{}{1/2}
\draw(1.5,0.5)node{$2$}(3.5,.5)node{$0$}(5.2,0.5)node{$0$}
(6.2,0.5)node{$1$};
} &  
\tikzpic{[scale=0.3]
\draw[thick](0,0)--(2,-2)--(5,1)--(6,0)--(9,3)--(8,4)
(1,-1)--(5,3)--(6,2)--(9,5)--(8,6)
(1,1)--(2,0)(0,0)--(5,5)--(6,4)--(9,7)--(8,8)
(5,5)--(6,6)--(7,5)(6,6)--(8,8);
\draw[dotted](2,0)--(3,-1)(2,2)--(4,0)(3,3)--(5,1)--(6,2)--(7,1)
(4,4)--(5,3)--(6,4)--(8,2)(7,5)--(8,4)(7,7)--(8,6);
}
\\
(a) & (b) & (c)
\end{array}
\end{eqnarray*}
\caption{Bijection among a configuration of ballot strips, a label of 
the binary tree and a label of the link pattern ($\alpha=22111211$ 
and $m=2$).}
\label{fig5}
\end{figure}

It is easy to show that the correspondence above is bijective by 
reverting the procedure.
Therefore, for both Case A and Case B, we have 
\begin{prop}
Let $\alpha<\beta$ be paths in $\mathcal{P}_N$ with $\epsilon=+$.
There exists a bijection between labellings of $A(\beta/\alpha)$ 
and configurations of ballot strips in the skew Ferrers diagram 
$\lambda(\beta)/\lambda(\alpha)$ satisfying Rule I.
\label{propbijection}
\end{prop}

From the proof of Proposition~\ref{propbijection}, we also know that 
a label of a link pattern is bijective to a configuration $\mathcal{C}$
of ballot strips.
We define the weight of a label of a link pattern as the weight of 
corresponding configuration $\mathcal{C}$.
To prove theorems, it is enough to show that the weight of a label of a 
binary tree and that of the corresponding link pattern are equal.
We show this statement for Case A and B simultaneously.

Fix a labelling $L$ of a link pattern $\pi(\beta)$, whose weight is 
$\mathrm{wt}(L)$.
We increase by one the label associated with a link $p$, where a link 
means one of a pair $12$, a vertical line, and paired $1$'s.
We also assume that the obtained labelling is an allowed labelling 
on $\pi(\beta)$. 
Recall that the position of a link $p$ and number of unpaired $1$'s 
right to $p$ uniquely determined the length of the ballot strip 
$\mathcal{D}$ corresponding to $p$.
The contribution of this increment to the generating function is the 
term $\mathrm{wt}(\mathcal{D})\mathrm{wt}(L)$ in the generating function
$Q_{\alpha,\beta}^{X,I,+}$. 
This increment on the label is translated in the language of 
$A(\beta/\alpha)$ as follows.
Suppose that an edge $e$ in $A(\beta/\alpha)$ is associated with the link
$p$ and let $n(e)$ be the label of the edge $e$.
From the construction of the bijection in Proposition~\ref{propbijection},
we increase by one the labels of all descendants of $e$ in the tree.
It is clear that the number of all types of edges (with or without 
``+", ``e" or ``o") descending to $e$ determines the length of 
the ballot strip $\mathcal{D}$.
Together with the weight contribution of all descending edges 
(see Definitions~\ref{defngfbtA} and \ref{defngfbtB}), the contribution 
of this binary tree to the generating function is exactly 
$\mathrm{wt}(\mathcal{D})\mathrm{wt}(L)$.
This implies that two generating functions $Q_{\alpha,\beta}^{X,I,-}$ 
and $R_{\alpha,\beta}^X (X=A,B)$ are equal.

\bibliographystyle{amsplainhyper} 
\bibliography{biblio}
\end{document}